\documentclass[11pt]{amsart}
\usepackage{amssymb,amsmath, graphicx}

\theoremstyle{plain}
\newtheorem{thrm}{Theorem}[section]
\newtheorem{lemma}[thrm]{Lemma}
\newtheorem{prop}[thrm]{Proposition}

\newtheorem{rmrk}[thrm]{Remark}
\newtheorem{dfn}[thrm]{Definition}

\numberwithin{equation}{section}

\begin{document}
\newcommand{\SL}{\mathcal L^{1,p}( D)}
\newcommand{\Lp}{L^p( Dega)}
\newcommand{\CO}{C^\infty_0( \Omega)}
\newcommand{\Rn}{\mathbb R^n}
\newcommand{\Rm}{\mathbb R^m}
\newcommand{\R}{\mathbb R}
\newcommand{\Om}{\Omega}
\newcommand{\Hn}{\mathbb H^n}
\newcommand{\aB}{\alpha B}
\newcommand{\eps}{\epsilon}
\newcommand{\BVX}{BV_X(\Omega)}
\newcommand{\p}{\partial}
\newcommand{\IO}{\int_\Omega}
\newcommand{\bG}{\mathbb{G}}
\newcommand{\bg}{\mathfrak g}
\newcommand{\bz}{\mathfrak z}
\newcommand{\bv}{\mathfrak v}
\newcommand{\Bux}{\mbox{Box}}
\newcommand{\ep}{\epsilon}
\newcommand{\dep}{\Delta_{H,p}}
\newcommand{\nh}{\nabla_H}
\newcommand{\Gp}{G_{\Om,p}}
\newcommand{\HH}{\mathbb H^1}

\title[Boundary behavior for $p$-harmonic functions, etc.]{Boundary behavior $p$-harmonic functions in  the Heisenberg group}

\author{Nicola Garofalo}
\address{Department of Mathematics\\Purdue University \\
West Lafayette, IN 47907} \email[Nicola
Garofalo]{garofalo@math.purdue.edu}
\thanks{First author supported in part by NSF Grant DMS-0701001}

\author{Nguyen Cong Phuc}
\address{Department of Mathematics \\
Louisiana State University\\
Baton Rouge, LA 70803-4918} \email[Nguyen Cong Phuc]{pcnguyen@math.lsu.edu}
\thanks{Second author supported in part by NSF Grant DMS-0901083}

\date{\today}

%
%
\keywords{}
\subjclass{}

\maketitle

\section{Introduction}\label{S:intro}

\vskip 0.2in

The study of the boundary behavior of nonnegative solutions of elliptic and parabolic pde's is a central subject with a long and rich history, see e.g. \cite{Ke}. Over the past few years the attention of several groups of workers in the field has increasingly focused on boundary value problems for a class of second order partial differential equations which arise from problems in geometry, several complex variables and also, and pre-eminently, in the applied sciences (e.g., robotics, neuroscience, financial mathematics). The relevant pde's, known as \emph{subelliptic equations}, display many challenging new aspects and typically, they fail to be elliptic at every point. 

Some interesting progress in the analysis of the boundary behavior of solutions to these equations has come with the works \cite{NS}, \cite{Citti}, \cite{D}, \cite{CG}, \cite{LU}, \cite{CGN2}, \cite{MM1}, \cite{MM2}, \cite{CGN3}. The prototypical situation is that of a graded nilpotent Lie group $\bG$ with a fixed sub-Laplacian
$\mathcal L = \sum_{i=1}^m X_i^2$ associated to an orthonormal basis
of the bracket generating layer of the Lie algebra. Given an open set
$\Om\subset \bG$, a distribution $u$ is called harmonic if $\mathcal L
u = 0$ in $\mathcal D'(\Om)$. By H\"ormander's hypoellipticity
theorem \cite{H} one has that every harmonic function is in fact
$C^\infty(\Om)$. Similarly to its classical counterpart a central problem is that of understanding the boundary behavior of nonnegative solutions of $\mathcal L u = 0$ in a given bounded open set $\Om \subset \bG$.  

This question poses remarkable new challenges with respect to the classical elliptic theory. On one hand, the vector fields $X_1,...,X_m$ satisfy nontrivial commutation relations and the commutators are, effectively, derivatives of higher order. This is reflected in the fact that the natural geometry attached to $\mathcal L$ is not the Riemannian geometry of the ambient manifold $\bG$, but a much more complicated nonisotropic one in which different directions in the tangent space weight in a different way, according to their order of commutation. A remarkable new aspect is then the interplay between the sub-Riemannian geometry associated with $\mathcal L$ and the nature of the boundary of the domain $\Om$. In this connection, those points of $\p \Om$ at which the vector fields $X_1,...,X_m$ become tangent to $\p \Om$ play a special role. At such points, which are known as \emph{characteristic points}, a hoist of new phenomena can occur and a solution of $\mathcal L u = 0$ can display behaviors quite different from classical harmonic functions.  

The subtle role of characteristic points first became apparent in the pioneering works of Fichera \cite{Fic1}, \cite{Fic2} (who first introduced the notion of characteristic set), Kohn-Nirenberg \cite{KN} and Bony \cite{B}. In his work  \cite{Jer}  on the Dirichlet problem in the Heisenberg group, D. Jerison first constructed an example of a smooth  (in fact, real analytic) domain for which the Dirichlet problem admits a Green function which, in the neighborhood of a characteristic point, is at most H\"older continuous up to the boundary, see also \cite{Jer2}. This is in sharp contrast with the classical elliptic theory, in which smooth data on smooth domains produce solutions which are smooth up to the boundary. 

In the papers \cite{CGN2}, \cite{CGN3} a complete solution of the Dirichlet problem was obtained for the class of the so-called ADP domains, i.e., domains which are \emph{admissible for the Dirichlet problem}. Such domains are defined by imposing that they be NTA (non-tangentially accessible) with respect to the sub-Riemannian distance associated with the vector fields $X_1,...,X_m$, and that furthermore they satisfy a uniform outer tangent ball condition reminiscent of that introduced by Poincar\'e in the classical setting \cite{P}. This second assumption was imposed to rule out D. Jerison's negative phenomenon mentioned above. For an extensive discussion of the various ramifications of these hypothesis we refer the reader to the paper \cite{CGN3}. The reader is also referred to Section \ref{S:nta} for the definitions of NTA, and ADP domains.

In this paper we initiate the study of the boundary behavior of nonnegative $p$-harmonic functions, i.e., weak solutions to the nonlinear equation 
\begin{equation}\label{pharm}
\mathcal L_p u \overset{def}{=} \sum_{i=1}^m X_i(|X u|^{p-2} X_i u) = 0,\ \ \
1<p<\infty.
\end{equation}
Equation \eqref{pharm} arises as the Euler-Lagrange equation of the $p$-energy in the Folland-Stein Sobolev embedding in \cite{Fo} and, similarly to its classical counterpart, the Euclidean $p$-Laplacian, it plays an important role in the analysis of sub-Riemannian spaces. In \eqref{pharm} we have indicated with $|Xu| = (\sum_{i=1}^m (X_iu)^2)^{1/2}$ the length of the sub-gradient of $u$.

The relevant geometric setting of the present paper is the
Heisenberg group $\Hn$, which is the simplest and perhaps most important example of a graded nilpotent Lie group of step two. For a detailed description of such group we refer the reader to Section \ref{S:cg}. There is a reason for which we confine our analysis to $\Hn$, rather then considering an arbitrary nilpotent Lie group, or even
more general settings, as it is done for instance in \cite{CGN3}.
Such reason will become apparent to the reader in the main body of the
paper, and is related to the considerable difficulties
connected with: 1) the nonlinear nature of \eqref{pharm}; 2) the present lacking of those tools which, in the classical setting, play an essential role in the analysis of the $p$-Laplacian. 

We will be primarily interested in the range $1<p\leq Q$, where $Q = 2n + 2$
is the homogeneous dimension of the Lie group $\Hn$, attached to the
non-isotropic dilations \eqref{dilHn} below. The range $p>Q$ is
also of interest in view of possible applications to the horizontal
$\infty$-Laplacian
\[
\mathcal L_\infty u = \sum_{i,j=1}^m u_{,ij} X_iu X_ju,
\]
where $u_{,ij} = \frac{1}{2}(X_iX_j u + X_j X_i u)$. However, in such range the proofs are completely analogous, perhaps a bit easier due to
the non singular nature of the corresponding fundamental solution
(incidentally, in view of the Morrey type theorem in \cite{GN}, for the domains in this paper functions in the
horizontal Sobolev space $W^{1,p}_H$ are automatically H\"older
continuous up to the boundary).

Our main objective is understanding, in the model situation of $\Hn$, the boundary behavior of those nonnegative  weak solutions of
\eqref{pharm} which continuously vanish on a portion of the boundary of a relevant domain $\Om$. In this perspective, our main contributions can be summarized as follows: 
\begin{itemize}
\item[1)] Theorem \ref{T:eaADP}, in which we obtain an estimate from above which says that any such solution should vanish at most linearly like the sub-Riemannian distance from the boundary:
\begin{equation}\label{ea0}
\frac{u(g)}{u(A_r(g_0))} \leq C^{-1}  \frac{d(g,\p \Om)}{r},
\end{equation}
where $A_r(g_0)\in \Om$ is a non-tangential point relative to $g_0\in \p \Om$.
\item[2)] Theorem \ref{T:eb}, in which we establish an estimate from below which states that the order of vanishing is exactly linear, i.e:
\begin{equation}\label{eb0}
\frac{u(g)}{u(A_r(g_0))} \ge C  \frac{d(g,\p \Om)}{r}.
\end{equation}
\end{itemize}
Combining these two results we obtain
\begin{equation}\label{le}
C  \frac{d(g,\p \Om)}{r} \le \frac{u(g)}{u(A_r(g_0))} \leq C^{-1}  \frac{d(g,\p \Om)}{r}.
\end{equation}
Finally, since the constant $C>0$ is independent of the particular $p$-harmonic function $u$, we conclude in Theorem \ref{T:cp} that for any two nonnegative $p$-harmonic functions $u, v$, which continuously vanish on a portion of the boundary, one has
 \begin{equation}\label{cpi}
 C \frac{u(A_r(g_0)}{v(A_r(g_0))} \le \frac{u(g)}{v(g)} \leq C^{-1} \frac{u(A_r(g_0)}{v(A_r(g_0))}.
 \end{equation}
Thus, \emph{all nonnegative $p$-harmonic functions which vanish on a portion of the boundary, must do so at the same rate}.

This description clearly provides an oversimplified picture, since we have not specified under which assumptions, and where, each of the relevant estimates \eqref{ea0}, \eqref{eb0} is valid. In this respect, it is worth observing that, in the Euclidean setting and for $p=2$, although the comparison theorem \eqref{cpi} does hold
for large classes of domains with rough boundaries (for instance, in Lipschitz
or even NTA domains, see \cite{CFMS}, \cite{JK}), the linear decay estimate
\eqref{le} breaks down if the domain fails to satisfy a uniform bound on its
curvatures. 

For instance, if $0<\theta_0 < \pi/2$ and we consider in $\R^2$
the convex circular sector $\Om = \{(r,\theta)\mid 0<r<1 ,
|\theta|<\theta_0\}$, where $\theta$ indicates the angle formed by
the directional vector of the point $(x,y)$ with the positive
direction of the $y$-axis, the function $u(r,\theta) = r^\lambda
\cos (\lambda\theta)$ is a nonnegative harmonic function in $\Om$
vanishing on that portion of $\p \Om$ corresponding to $|\theta| =
\theta_0$ provided that\[ \lambda  = \frac{\pi}{2\theta_0}.
\]
From our choice, we have $\lambda >1$ and therefore this example
shows that for domains without an interior tangent ball the estimate
from below \eqref{eb0} cannot possibly hold in general. Using the
same type of domain and function, but this time with
$\pi/2<\theta_0<\pi$ (a non-convex cone) we see that if the tangent
outer ball condition fails, then there exist harmonic functions
which vanish at the boundary at best with a H\"older rate $<1$.
Therefore, the estimate from above \eqref{ea0} cannot possibly
hold in general.

In the classical setting of $\Rn$, when the domain satisfies a uniform tangent ball condition,  i.e., it possesses at every boundary point a ball tangent from the inside, and one tangent from the outside, with a uniform control on the radii of such balls, the linear decay estimate \eqref{le} does hold, even for solutions to uniformly elliptic or parabolic equations, and its proof can be found in \cite{G}. Recently, this result has been extended to the classical $p$-Laplacian in $\Rn$ in the paper \cite{AKSZ}. Now, it is well-known that the uniform tangent ball condition characterizes $C^{1,1}$ domains, and from the above examples it is clear that this degree of smoothness is essentially optimal if one looks for a linear decay such as that in \eqref{le}. This introduces us to the central theme of the present paper.

In the study of the sub-Riemannian Dirichlet problem the (Euclidean) smoothness of the ground domain has no bearing on the boundary behavior of the relevant harmonic functions. This is due to the presence of characteristic points on the boundary. To illustrate this point, consider the real-analytic domain  $\Om_M = \{(z,t)\in
\Hn\mid t> - M |z|^2\}$, which possesses an isolated characteristic point $e = (0,0)$. Then, in \cite{Jer}, part II, it was proved that there exists $M>0$ such that
$\Om_M$ supports a nonnegative harmonic function (a solution of the
real part of the Kohn-Spencer sub-Laplacian in $\Hn$) which vanishes on the boundary,
and which goes to zero near $e = (0,0)\in \p \Om_M$ at most like
$d((z,t),\p \Om_M)^\lambda$, for some $0<\lambda<1$. Therefore, for this $\Om_M$ an estimate such as \eqref{ea0} fails (the reader should also notice
that this example shows the failure of Schauder type estimates at
characteristic points). We note here that, from the point of view of the sub-Riemannian geometry of $\Hn$ (the dilations of $\Hn$ are the non-isotropic dilations $(z,t)\to (\lambda z,\lambda^2 t)$), near its characteristic point $e$ the paraboloid $\Om_M$ looks like the (Euclidean) non-convex cone discussed above and, in fact, such domain fails to satisfy the intrinsic
tangent outer ball condition at $e$.

But the situation is even worse than this! The domain $\Om
=\{(z,t)\in \Hn\mid |z|^4 + 16 (t-1)^2<1\}$ is a smooth domain (in fact, real analytic) which
satisfies the uniform interior (and
exterior) ball condition at every point of its boundary (for this see \cite{CG}). In such domain for every
$1<p<\infty$ the function
$u(z,t) = t$ is a positive $p$-harmonic function  which vanishes at the (characteristic) boundary point $e =
(0,0)\in \p \Om$. But for the sub-Riemannian distance $d$ of $\Hn$ we have for any $0<t<1$
\[
d((0,t),e)) \cong \sqrt{t},
\]
and therefore $u((0,t)) \cong d((0,t),\p \Om)^2$. Thus, the linear estimate from below \eqref{eb0} fails for this example.

It should be clear from this discussion that: 1) If one hopes for an estimate from above such as \eqref{ea0} a condition such as an intrinsic uniform outer ball must be imposed. The interesting aspect of this geometric assumption is that it does not distinguish between characteristic or non-characteristic points. In the sense that, once it is assumed, the twisted geometry of the sub-Riemannian balls at characteristic points automatically rules out negative examples such as the above domain $\Om_M$.  2) The situation for the estimate from below \eqref{eb0} is quite more difficult. As we have seen, it cannot hold on the characteristic set, but it is not clear for which domains it does hold. 
There is also a third important aspect which pertains both estimates \eqref{ea0}, \eqref{eb0}. One needs to have enough regions of non-tangential approach to the boundary, and for this one needs to assume that the relevant domain be NTA (non-tangentially accessible) with respect to the sub-Riemannian metric.

We are now in the position of presenting the precise statement of \eqref{ea0}.

\begin{thrm}\label{T:eaADP}
Let $\Om\subset \Hn$ be an ADP domain and let $g_0\in \p \Om$,  $0<r<R_0/6$ where $R_0>0$ depends
only on the ADP character of $\Om$. If  $u$ is a nonnegative $p$-harmonic
function in $\Om\cap B(g_0,6r)$ which vanishes continuously on
$\partial\Om\cap B(g_0,6r)$, then there exists $C = C(n,\Om,p)>0$ such that
for every $g\in \Om\cap B(g_0,r)$ one has
\begin{equation}\label{eai}
\frac{u(g)}{u(A_r(g_0))} \leq C  \frac{d(g,\p \Om)}{r}.
\end{equation}
\end{thrm}

As we have mentioned above, ADP domains are defined by imposing that they be NTA with respect to the sub-Riemannian distance associated with the vector fields $X_1,...,X_m$, and that furthermore they satisfy a uniform outer tangent ball condition, see Definition \ref{D:uob}. It is worth emphasizing here that in $\Hn$ the class of such domains is very rich. For instance, every (Euclidean) convex and $C^{1,1}$ domain is ADP. This follows from the fact that (Euclidean) $C^{1,1}$ domains are NTA, see \cite{MM2} and Theorem \ref{T:MM} below, and that (Euclidean) convex domains possess a uniform gauge ball tangent from the outside, see \cite{LU}. Concerning \eqref{eai} we note that, by the Harnack inequality in \cite{CDG1}, we know that $u>0$ in $\Om$, and thus it makes sense to divide by $u(A_r(g_0))$.

We now turn our attention to the estimate from below \eqref{eb0}. The implementation of the ideas in \cite{G} in the sub-Riemannian setting involves a delicate analysis whose central objective is proving the existence of appropriate uniform families of intrinsic balls which are tangent from the inside to the relevant domain, and whose centers are located along paths which possess a crucial segment, or quasi-segment property with respect to the sub-Riemannian distance. It is remarkable that away from the characteristic set, at every scale, (Euclidean) $C^{1,1}$ domains possess such uniform families of balls. Proving this fact involves a substantial amount of work, and it constitutes the entire content of Sections \ref{S:qs1} and \ref{S:qs2}. 
As a consequence of such work we show that every (Euclidean) $C^{1,1}$ domain $\Om\subset \Hn$ satisfies the uniform ball condition away from its characteristic set, in the sense of Definition \ref{D:moduob} below.  As we have mentioned, on the characteristic set this delicate geometric property fails even if $\p \Om$ is real analytic, but, on the other hand, we have seen that the estimate \eqref{eb0} fails as well.

\begin{thrm}\label{T:eb}
Let $u$ be a nonnegative $p$-harmonic function
in a bounded (Euclidean) $C^{1,1}$ domain $\Om\subset \Hn$. Then, there exists $M>1$ depending only on $\Om$ such that for every
$g_0\in \p \Om\setminus \Sigma_\Om$ and every
$0<r<\frac{d(g_0,\Sigma_\Om)}{M}$, one has for some constant $C =
C(n,\Om,p)>0$
\begin{equation}\label{ebi}
\frac{u(g)}{u(A_r(g_0))} \geq C \frac{d(g,\p \Om)}{r}
\end{equation}
for every $g\in \Om \cap B(g_0,r)$. \end{thrm}

In the above statement we have indicated by $\Sigma_\Om$ the characteristic set of $\p \Om$. This is a compact subset of $\p \Om$. We emphasize that, as far as we are aware of, Theorem \ref{T:eb} is new even in the linear case $p=2$. Finally, by combining Theorems \ref{T:eaADP} and \ref{T:eb} we obtain the
following result.

\begin{thrm}[Boundary comparison principle]\label{T:cp}
Let $\Om\subset \Hn$ be a bounded (Euclidean) $C^{1,1}$ domain.
Given
$g_0\in \p \Om\setminus \Sigma_\Om$,  for
$0<r<\frac{d(g_0,\Sigma_\Om)}{M}$, where $M>1$ depends only on $\Om$,  let $u$, $v$ be two nonnegative $p$-harmonic
functions in $\Om \cap B(g_0,6r)$ vanishing
continuously on $\partial\Om\cap B(g_0,6r)$. There exists $C = C(n,\Om,p)>0$
such that for every $g\in \Om \cap B(g_0,r)$ one has
\begin{equation}\label{ctsub}
C \frac{u(A_r(g_0)}{v(A_r(g_0))} \le \frac{u(g)}{v(g)} \leq C^{-1} \frac{u(A_r(g_0)}{v(A_r(g_0))}.
\end{equation}
\end{thrm}
 
It is worth mentioning here that, because of the nature of our approach, we obtain \eqref{ctsub} only away from the characteristic set. On the other hand, in \cite{CG} it was proved that, in the linear case $p=2$, a result like Theorem \ref{T:cp} does hold in every NTA domain and for general sub-Laplacians associated with smooth vector fields satisfying H\"ormander's finite rank condition. We also mention the recent paper \cite{LN} in which the authors have established a result such as \eqref{ctsub} for the classical $p$-Laplacian in a Lipschitz domain in $\Rn$. 
Whether in the nonlinear setting of the present paper Theorem \ref{T:cp} can be extended on the characteristic set remains at the moment a very challenging direction of investigation, which we defer to a future study.

In connection with Theorem \ref{T:eaADP} we also have the following result which provides a sharp estimate at the boundary for the Green function associated with the nonlinear operator in \eqref{pharm}. In the linear case $p=2$, such estimate was first obtained for the Heisenberg group $\Hn$ in \cite{LU}, and it was generalized to groups of Heisenberg type in \cite{CGN2}, and to general operators of H\"ormander type in \cite{CGN3}.
We recall that, in such general setting, it was proved in \cite{CDG2} that the fundamental solution $\Gamma_p$ of quasilinear equations modeled on \eqref{pharm} satisfies the following asymptotic estimate near its singularity
\[
C \left(\frac{d(g,g')^p}{|B(g,d(g,g'))|}\right)^{\frac{1}{p-1}} \le \Gamma_p(g,g') \le C^{-1} \left(\frac{d(g,g')^p}{|B(g,d(g,g'))|}\right)^{\frac{1}{p-1}}.
\]
If now $\Om$ is a bounded open set, then from this estimate and the weak maximum principle one immediately derives that the $p$-Green function $G_{\Om,p}$ for $\Om$ must satisfy the same control from above, i.e.,
\[
G_{\Om,p}(g,g') \le C^{-1} \left(\frac{d(g,g')^p}{|B(g,d(g,g'))|}\right)^{\frac{1}{p-1}},
\]
for all points $g, g'\in \Om$ with $g\not=g'$. However, for a domain $\Om \subset \Hn$ which satisfies the uniform outer ball condition, we obtain the following sharp result. 

\begin{thrm}\label{T:pimproved}
Let $\Om\subset \Hn$ be a bounded domain satisfying the uniform
outer ball condition. Given $1<p\leq
Q$, let $G_{\Om,p}$ denote the Green function associated with
\eqref{pharm} and $\Om$.
\begin{itemize}
\item[(i)] If $1<p<Q$ there exists a constant $C = C(n,\Om,p)>0$ such that
\[
G_{\Om,p}(g',g) \leq C
\left(\frac{d(g,g')}{|B(g,d(g,g'))|}\right)^{1/(p-1)} d(g',\p \Om),\ \ g,g'\in \Om,\ g'\not=g.
\]
\item[(ii)] If $1<p<Q$ and $p$ and $\Om$ are such that the $\Gp(g',g) =
\Gp(g,g')$, then one has
\[
G_{\Om,p}(g,g') \leq C
\left(\frac{d(g,g')}{|B(g,d(g,g'))|}\right)^{1/(p-1)} \frac{d(g,\p
\Om) d(g',\p \Om)}{d(g,g')},\ \ g,g'\in \Om,\ g\not= g'.
\]
\item[(iii)] If $p=Q$, then
\[
G_{\Om,p}(g',g) \leq C \log
\left(\frac{diam(\Om)}{d(g,g')}\right) \frac{d(g',\p
\Om)}{d(g,g')},\ \ g,g'\in \Om,\ g'\not=g,
\]
\item[(iv)] When $p=Q$ and $\Gp(g',g) = \Gp(g,g')$ one has
\[
G_{\Om,p}(g',g)\ \leq\ C\ \log
\left(\frac{diam(\Om)}{d(g,g')}\right) \frac{d(g,\p \Om)d(g',\p
\Om)}{d(g,g')^2},\ \ g,g'\in \Om,\ g'\not=g.
\]
\end{itemize}
\end{thrm}

\begin{rmrk}\label{R:nonsym} Concerning parts (ii) and (iv) we mention that to the best of our knowledge the
question of symmetry (or non-symmetry) of the Green function is
largely unsettled even in the classical case of $\Rn$ and
of the standard $p$-Laplace equation. Using conformal
invariance, Janfalk \cite{Ja} has shown that the Green function for
the unit ball and the $n$-Laplacian is symmetric when $n>2$. He also
proved that for the same domain given any $x\not= 0$ there exists a
$p_x>n$ such that $\Gp(x,0)\not= \Gp(0,x)$ for all $p>p_x$. We are
not aware of results in either direction when $1<p<n$.
\end{rmrk}

In the linear case $p=2$ treated in \cite{CGN2}, \cite{CGN3} the sharp estimate (i) in Theorem \ref{T:pimproved} was used, in
combination with several basic harmonic analysis results obtained in
\cite{CG}, and with a crucial Ahlfors' type estimate for the
horizontal perimeter, to prove that for a ADP domain
the subelliptic
Poisson kernel satisfies a reverse H\"older inequality. As a
consequence of this fact, it was shown that harmonic measure, the
horizontal perimeter measure, and the standard surface measure are
mutually absolutely continuous. Furthermore, the Dirichlet problem
was solved for boundary data in $L^p$ with respect to the surface
measure. We plan to address some of these questions in the nonlinear setting of this paper in a   future study.

\section{Preliminaries}\label{S:cg}

The simplest and most important example of a non-Abelian Carnot
group of step $r=2$ is the $(2n+1)$-dimensional Heisenberg group
$\Hn$. We recall that a Carnot group of step $r$ is a connected, simply connected Lie
group $\bG$ whose Lie algebra $\bg$ admits a stratification $\bg=
V_1 \oplus \cdots \oplus V_r$ which is $r$-nilpotent, i.e.,
$[V_1,V_j] = V_{j+1},$ $j = 1,...,r-1$, $[V_j,V_r] = \{0\}$, $j =
1,..., r$. A trivial example of (an Abelian) Carnot
group is $\bG = \Rn$, whose Lie algebra admits the trivial
stratification $\bg = V_1 = \Rn$. The prototype \emph{par excellence} of a non-Abelian Carnot group is the Heisenberg group $\Hn$.

To describe such group it will be convenient to identify the generic
point $x+ iy\in \mathbb C^n$ with $z = (x,y)\in \R^{2n}$. With such identification $\mathbb C^n
\times \R$ is identified with $\R^{2n+1}$, and henceforth we denote with $g =
(x,y,t), g'=(x',y',t')$, etc., generic points in $\R^{2n+1}$. For a given
$z=(x,y)\in \R^{2n}$, we will denote $z^\perp = (y,-x)$. Notice that $z^\perp = J z$, where $J$ is the simplectic matrix in $\R^{2n}$ 
\begin{equation}\label{J}
J = \begin{pmatrix} 0 & I \\ - I & 0 \end{pmatrix}.
\end{equation}
The
Heisenberg group $\Hn$ is the Lie group whose underlying manifold
(in real coordinates) is $\R^{2n+1}$ with non-Abelian group
multiplication
\begin{eqnarray}\label{hgl}
g\  g'\ & =&\ (x,y,t)\  (x',y',t')
\\
& =&\ \left(x + x', y + y', t + t' + \frac{1}{2} (<x,y'> -
<x',y>)\right)\nonumber\\
& =&\ \left(z+z',t+t' +\frac{1}{2}<z,(z')^\perp>\right)\ .\nonumber
\end{eqnarray}
We will indicate with $e=(0,0,0)\in \Hn$ the group identity with
respect to \eqref{hgl}. Notice that for a given $g = (x,y,t)$ one has $g^{-1} = (-x,-y,-t)$. We let $L_g(g') = g g'$ denote the operator
of left-translation on $\Hn$, and indicate with $(L_g)_*$ its
differential. The Heisenberg algebra admits the decomposition
$\mathfrak h_n = V_1\oplus V_2$, where $V_1 = \mathbb R^{2n} \times
\{0\}_t$, and $V_2 = \{0\}_{\mathbb R^{2n}}\times \R_t$. Identifying
$\mathfrak h_n$ with the space of left-invariant vector fields on
$\Hn$, one easily recognizes that a basis for $\mathfrak h_n$ is
given by the $2n+1$ vector fields
\begin{equation}\label{vf2}
\begin{cases}
(L_g)_*\left(\frac{\p}{\p x_i}\right)\ \overset{def}{=} X_i(g) =
\frac{\p}{\p x_i} - \frac{y_i}{2}\ \frac{\p}{\p t}, \qquad i=1, \dots, n,
\\
(L_g)_*\left(\frac{\p}{\p y_i}\right) \overset{def}{=} X_{n+i}(g)
= \frac{\p}{\p y_i} + \frac{x_i}{2}\ \frac{\p}{\p t}, \qquad i=1, \dots, n,
\\
(L_g)_*\left(\frac{\p}{\p t}\right) \overset{def}{=} T(g) =
\frac{\p}{\p t},
\end{cases}
\end{equation}
and that the only non-trivial commutation relation is
\begin{equation}\label{commHn}
[X_i,X_{n+j}] = T \delta_{ij}\ ,\quad\quad\quad\quad i , j =
1, \dots, n.
\end{equation}

The relation \eqref{commHn} shows that $[V_1,V_1] =
V_2$. Since, as we have said, $[V_1,V_2] = \{0\}$, the
Heisenberg group is a Carnot group of step $r=2$.

The subspace $V_1$ is called the horizontal layer, whereas $V_2$ is
called the vertical layer of the Heisenberg algebra. It is clear
that $V_2$ constitutes the center of $\mathfrak h_n$ with respect to \eqref{hgl}. Elements of
$V_j$, $j=1,2$, are assigned the formal degree $j$. The associated
non-isotropic dilations of $\Hn$ are given by
\begin{equation}\label{dilHn}
\delta_\lambda(g) = (\lambda x, \lambda y, \lambda^2 t).
\end{equation}
The \emph{homogeneous dimension} of $\Hn$ with respect to
\eqref{dilHn} is the number $Q = 2n + 2$. In the analysis of $\Hn$
such number plays much the same role as that of the Euclidean dimension of
$\Rn$. This is justified by the fact that, given that Lebesgue
measure $dg$ is a left- and right-invariant Haar measure on $\Hn$,
one easily checks that
\[
(d\circ\delta_\lambda)(g) = \lambda^Q dg .
\]

We denote by $d_{cc}(g,g')$ the \emph{CC (or Carnot-Carath\'eodory)
distance} on $\Hn$ associated with the system $X=\{X_1,\dots,
X_{2n}\}$. It is well-known that in the Heisenberg group there is
another distance equivalent to $d_{cc}(g,g')$. Consider in fact the
Kor\'anyi-Folland nonisotropic gauge on $\Hn$\[ N(g) = (|z|^4 + 16
t^2)^{1/4}. \] Then it was proved in \cite{Cy} that
\[
d(g,g') = N(g^{-1} g')
\]
defines a metric on $\Hn$, the so called \emph{gauge distance}. The following formula, which follows from \eqref{hgl} will be often useful in the rest of this paper
\begin{equation}\label{hgl2}
d(g,g') = \left\{|z'-z|^4 + 16 \left(t'-t + \frac{1}{2}<z', z^\perp>\right)^2 \right\}^{1/4}.
\end{equation}
For later use we will need the following lemma.

\begin{lemma}\label{L:inv}
Let $S\in U(n)$ be a unitary matrix. If for $g = (z,t)$ we denote $Sg = (Sz,t)$ the action of $U(n)$ on $\Hn$, then 
\begin{equation}\label{inv}
d(Sg,Sg') = d(g,g'),\ \ \ \ g, g'\in \Hn.
\end{equation}
\end{lemma}

\begin{proof}
Note that if $S\in O(2n)$ is an orthogonal matrix such that 
\begin{equation}\label{SJ}
SJ = JS,
\end{equation}
where $J$ is the symplectic matrix in \eqref{J}, then we have
\[
S z^\perp = (Sz)^\perp,
\]
and thus from \eqref{hgl2} we conclude that \eqref{inv} holds. Therefore, the transformations which leave the gauge distance in $\Hn$ invariant are those arising from matrices $S\in O(2n)$ which satisfy \eqref{SJ}.
Now, notice that for \eqref{SJ} to hold one must have
\[
S = \begin{pmatrix}
A & - B
\\
B & A\end{pmatrix}.
\]
The group of these matrices is the symplectic group $Sp_n(\R)$, and it is well-known that
\[
O(2n) \cap Sp_n(\R) = U(n),
\]
the unitary group, see for instance \cite{Be}, p. 24. 

\end{proof}

One
easily verifies that there exists $C = C(n)>0$ such that
\begin{equation}\label{equivalence}
C\ d(g,g')\ \leq\ d_{cc}(g,g')\ \leq\ C^{-1}\ d(g,g'),
\quad\quad\quad\quad g,g'\in \Hn.
\end{equation}

Both $d_{cc}$ and $d$ are left-invariant
\begin{equation}\label{isometry}
d_{cc}(L_g(g'),L_g(g''))\ =\ d_{cc}(g',g'')\ ,\quad\quad\quad\quad
d(L_g(g'),L_g(g''))\ =\ d(g',g'')\ ,
\end{equation}
and homogeneous of degree one
\begin{equation*}
d_{cc}(\delta_\lambda(g'),\delta_\lambda(g''))\ =\ \lambda\ d_{cc}(g',g'')\
,\quad\quad\quad\quad d(\delta_\lambda(g'),\delta_\lambda(g''))\
=\ \lambda\ d(g',g'')\ .
\end{equation*}

In view of \eqref{equivalence} we can use either one of the two distances in all metric questions. Since, unlike $d_{cc}$, the gauge distance is smooth, in this paper we will exclusively work with the latter. If $|E|$ indicates the Lebesgue measure of a set $E\subset \Hn$, then denoting with
\begin{equation*}
B(g,R) = \{g'\in \Hn\mid d(g',g)<R\} ,
\end{equation*}
the gauge ball centered at $g$ with radius
$R$, one easily recognizes that there exist $\alpha_n>0$ such that for every $g\in \Hn$, $R > 0$, 
\begin{equation*}
|B(g,R)|\ =\
\alpha_n R^Q.
\end{equation*}

\section{$p$-Harmonic functions}\label{S:pharm}

Let $\Om\subset \Hn$ be an open
set. For $1\leq p\leq \infty$ we indicate with
$W^{1,p}_H(\Om)$ the Folland-Stein Sobolev space of the functions
$f\in L^p(\Om)$ whose distributional horizontal derivatives $X_if\in
L^p(\Om)$ for $i=1,...,m$, where $m=2n$, and the vector fields $X_1,...,X_{2n}$ are defined by \eqref{vf2}. The spaces $W^{1,p}_{H,loc}(\Om)$ and
${\overset{o}{W}}^{1,p}_H(\Om)$ are defined similarly to the
classical ones. Given $1<p<\infty$ we say that $u\in
W^{1,p}_{H,loc}(\Om)$ is $p$-harmonic in $\Om$ if
\[
\int_\Om |X u|^{p-2}<X u,X \phi> dg = 0,
\]
for every $\phi\in \overset{o}{W}^{1,p}_H(\Om)$ such that
$supp(\phi)\subset \subset \Om$. From the results in \cite{CDG1} it
is known that $p$-harmonic functions can be redefined on a set of
measure zero so that they become $\alpha$-H\"older continuous for
some $\alpha = \alpha(n,p)\in (0,1)$. Furthermore, nonnegative
$p$-harmonic functions satisfy the uniform Harnack inequality, see
\cite{CDG1}. For the following results we refer the reader to \cite{D}.

\begin{thrm}[Existence and uniqueness in the Dirichlet problem]\label{T:diri}
Let $\Om\subset \Hn$. Given any $\phi\in W^{1,p}_H(\Om)$, there
exists a unique $p$-harmonic function $u=H^{\Om,p}_{\phi}\in W^{1,p}_H(\Om)$ such
that $u - \phi \in {\overset{o}{W}}^{1,p}_H(\Om)$.
\end{thrm}

We also have the following.

\begin{thrm}[Comparison principle]\label{T:ct}
Let  $u\in W^{1,p}_{H,loc}(\Om)$  be a $p$-superharmonic function and $v\in W^{1,p}_{H,loc}(\Om)$ be a $p$-subharmonic
function in $\Om$. If $min\{u-v,0\}\in
{\overset{o}{W}}^{1,p}_H(\Om)$, then $u\geq v$ a.e. in $\Om$.
\end{thrm}

Given a bounded open set $\Om \subset \mathbb{H}^n$ a point $g_0\in \p \Om$
is called \emph{regular} for the Dirichlet problem if for every
$\phi\in W^{1,p}_H(\Om)\cap C(\overline \Om)$, one has
\[
\underset{g\to g_0}{\lim} H^{\Om,p}_\phi(g)\ =\ \phi(g_0)\ .
\]

If every $g_0\in \p \Om$ is regular, then we say that $\Om$ is
regular for the Dirichlet problem. A basic Wiener type
estimate was proved in \cite{D}. From such result it follows that a sufficient geometric condition for $\Om$ to be
regular is that its exterior have uniform positive density. This means that
there exist $C>0$ and $R_0>0$ such that for every $g_0\in \p \Om$,
and $0<r<R_0$ one has
\begin{equation}\label{EC}
|\Om^c \cap B(g_0,r)| \geq C r^Q. \end{equation}

In fact, from the main result in \cite{D} one can infer that \eqref{EC} actually implies the H\"older continuity up to the boundary (with respect to the distances $d$ or $d_{cc}$ in \eqref{equivalence}) of the solution to the Dirichlet problem. For instance, any non-tangentially accessible (NTA) domain with respect to either one of the distances $d$ or $d_{cc}$ possesses a uniform exterior non-tangential point attached to every boundary point (for the notion of NTA domain see Definition \ref{D:NTA} below). This implies that any such domain satisfies \eqref{EC}, and therefore it is regular for the Dirichlet problem and the solution to such problem is in fact H\"older continuous up to the boundary. For a detailed study of the Dirichlet problem in NTA domains in the linear case $p=2$ we refer the reader to \cite{CG}. For the purpose of this paper the reader should keep in mind that in $\Hn$ every bounded domain whose boundary is Euclidean $C^{1,1}$ is NTA, and therefore satisfies \eqref{EC}. This interesting result was proved (in fact, for every Carnot group of step $r=2$), in \cite{MM2}, see Theorem \ref{T:MM} below.

\section{Singular solutions}\label{S:ss}

Let $1<p<\infty$. A distribution $\Gamma_p(\cdot,g)$ is called a
\emph{fundamental solution} of \eqref{pharm} with singularity at $g\in \mathbb{H}^n$ if: (i)
$\Gamma_p(\cdot,g)\in W^{1,p}_{H,loc}(\Hn\setminus\{g\})$;  (ii) $|X
\Gamma_p(\cdot,g)|^{p-1}\in L^1_{loc}(\Hn)$; and (iii)
\[ \int_{\Hn} |X \Gamma_p(\cdot,g)|^{p-2}<X
\Gamma_p(\cdot,g),X \phi> dg' = \phi(g), \]
 for every $\phi\in C^\infty_0(\Hn)$. We will need the following special case of a basic result from \cite{CDG2} (for the case $p=Q$ see also \cite{HH}).

\begin{thrm}\label{T:pfs}
For every $1<p<\infty$ the
function
\begin{equation}\label{pfs1}
\Gamma_p(g,g') = \Gamma_p(g',g) =
\begin{cases}
\frac{p-1}{Q-p} \sigma_p^{-\frac{1}{p-1}}
d(g,g')^{\frac{p-Q}{p-1}},\ \ p\not= Q, \\
\\
-\ \sigma_p^{-\frac{1}{p-1}} \log\ d(g,g'),\ \ \ p = Q,
\end{cases}
\end{equation}
with $g'\not=g$, is a fundamental solution of \eqref{pharm} with
singularity at $g\in \Hn$.
\end{thrm}

In \eqref{pfs1} we have let $\sigma_p = Q \omega_p$, where
\[
\omega_p = \int_{B(e,1)} |X d(\cdot, e)|^p dg.
\]

\begin{dfn}\label{D:green}
Given a bounded domain $\Om \subset \Hn$ we say that a distribution
$G_{\Om,p}(\cdot, g) \geq 0$ is a \emph{Green function with singularity at
$g\in \Om$} for \eqref{pharm} if: (i) $G_{\Om,p}(\cdot,g)\in
W^{1,p}_{H,loc}(\Om\setminus\{g\})$; (ii) $|X
\Gp(\cdot,g)|^{p-1}\in L^1(\Om)$; (iii) $\phi\ \Gp(\cdot,g)\in
\overset{o}{W}^{1,p}_H(\Om)$, for any $\phi \in C^\infty_0(\Hn)$
such that $\phi \equiv 0$ in a neighborhood of $g$; and (iv)
\[ \int_{\Om} |X \Gp(\cdot,g)|^{p-2}<X
\Gp(\cdot,g),X \phi> dg' = \phi(g), \]
 for every $\phi\in C^\infty_0(\Om)$.
\end{dfn}

In \cite{DG} it was proved, among other things, that given any
regular bounded open set $\Om\subset \mathbb{H}^n$, there exists a Green
function $\Gp(\cdot,g)$ with singularity at $g\in \Om$. It was also
shown in \cite{DG} that a Green function satisfies an asymptotic
estimate near the singularity similar to that in Theorem \ref{T:pfs}.
From such estimate and the comparison
principle (Theorem \ref{T:ct}) we conclude that there exists a
constant $C(n,\Om,p)>0$ such that for every $g'\in \Om$, $g'\not=
g$,
\begin{equation}\label{gg} \Gp(g',g) \leq \begin{cases}
C \ d(g,g')^{(p-Q)/(p-1)},\ \ 1<p<Q,
\\
-C  \log d(g,g'),\ p=Q.
\end{cases}
\end{equation}

\section{NTA and ADP domains}\label{S:nta}

We now introduce the relevant classes of domains for the results in this paper. We begin with recalling a geometric condition introduced in \cite{LU}, \cite{CGN1}, \cite{CGN2}, \cite{CGN3}, which is reminiscent of the classical outer sphere condition of Poincar\'e \cite{P}. We recall that the notation $B(g,r)$ indicates the non-isotropic gauge ball with respect to the distance \eqref{hgl2}.
 
\begin{dfn}\label{D:uob}
We say that a bounded domain $\Om \subset \Hn$ satisfies the \emph{uniform
outer (interior) ball condition} if there exists $R_0>0$ such that
for every $g_0\in \p \Om$ and every $0<r<R_0$, there exists
$B(g_1,r) \subset \Hn\setminus \overline \Om$ ($B(g_1,r) \subset
\Om$) for which $g_0\in \p B(g_1,r)$. If $\Om$ satisfies both the
uniform outer and interior ball conditions, then we say that $\Om$
satisfies the \emph{uniform ball condition}.
\end{dfn}

We emphasize that, as we have mentioned in Section \ref{S:intro}, in $\Hn$ there exist real analytic domains, such as for instance $\Om = \{(z,t)\in \Hn\mid t > - M |z|^2\}$, for any fixed $M>0$, which fail to satisfy the outer tangent ball condition. In this particular case, the domain $\Om$ does not possess an outer tangent ball at the characteristic point $e = (0,0)$. In fact, in view of the parabolic dilations $\lambda \to (\lambda z,\lambda^2 t)$ of $\Hn$, from the geometric viewpoint of $\Hn$ the domain $\Om$ looks like a non-convex cone.  

For some of the results in this paper the uniform outer ball condition, or the uniform ball condition will not be needed on the whole of the boundary of a given domain, but just away from its characteristic set.  

\begin{dfn}\label{D:moduob}
We say that a bounded domain $\Om \subset \Hn$ satisfies the \emph{uniform
outer (interior) ball condition away from its characteristic set} $\Sigma_\Om$ if there exists $\epsilon>0$ such that
for every $g_0\in \p \Om\setminus \Sigma_\Om$ and every $0<r<\epsilon \, d(g_0,\Sigma_\Om)$, there exists
$B(g_1,r) \subset \Hn\setminus \overline \Om$ ($B(g_1,r) \subset
\Om$) for which $g_0\in \p B(g_1,r)$. If $\Om$ satisfies both the
uniform outer and interior ball conditions away from $\Sigma_\Om$, then we say that $\Om$
satisfies the \emph{uniform ball condition away from its characteristic set}.
\end{dfn}

Next, we recall the class of NTA (non-tangentially accessible) domains. In the Euclidean setting such class was introduced in \cite{JK}. We emphasize that the definition of NTA domain is purely metrical, i.e. it  can
be formulated in an arbitrary metric space. In the framework of metrics associated with a system of smooth vector fields satisfying the finite rank condition, a detailed study of such domains was developed in \cite{CG}, and we refer the reader
to that source for all relevant results. One should also consult the paper \cite{AS} for further generalizations. Here, we will focus on the special
yet basic setting of $\Hn$ with its gauge
metric. First, given a bounded domain $\Om\subset \Hn$, a ball $B(g,r)$ is
called $M${\emph{-nontangential}} in $\Om$ if
\[
\frac{r}M<d(B(g,r),\p\Om)<Mr.
\]
Given two points $g,g'\in\Om$, a sequence of $M$-nontangential balls
in $\Om$, $B_1,\dots,B_k$, will be called a Harnack chain of length
$k$ joining $g$ to $g'$ if $g\in B_1$, $g'\in B_k$, and $B_i\cap
B_{i+1}\not=\phi$ for $i=1,\dots,k-1$. It should be noted that
consecutive balls have comparable radii.

\begin{dfn}\label{D:NTA}
We say that a bounded domain $ \Om\subset \Hn$ is a
\emph{nontangentially accessible domain} (NTA domain, hereafter) if
there exist $M$, $R_0>0$ for which:
\begin{itemize}
\item[(i)] (Interior corkscrew condition) For any $g_0\in\partial \Om$ and $0<r\leq R_0$
there exists $A_r(g_0)\in \Om$ such that
$\frac{r}M<d(A_r(g_0),g_0)\leq r$ and $d(A_r(g_0),\partial
\Om)>\frac{r}M$. (This implies that $B(A_r(g_0),\frac{r}{2M})$ is
$3M$-nontangential.)\\

\item[(ii)] (Exterior corkscrew condition) $ \Om^c$ satisfies
property (i).\\

\item[(iii)] (Harnack chain condition) For any $\epsilon>0$ and $g,g'\in \Om$ such
that $d(g,\partial \Om)>\epsilon$, $d(g',\partial \Om)>\epsilon$,
and $d(g,g')<C\epsilon$, there exists a Harnack chain joining $g$ to
$g'$ whose length depends on $C$ but not on $\epsilon$.
\end{itemize}
\end{dfn}

In \cite{CG} it was proved that in every Carnot group of step $r=2$
the gauge balls are NTA domains. Subsequently, in \cite{MM2} the authors proved the following interesting result.

\begin{thrm}\label{T:MM}
In every Carnot group of step two, and hence in particular in $\Hn$, every 
(Euclidean) $C^{1,1}$ domain is
NTA. Furthermore, such regularity is sharp as there exist
$C^{1,\alpha}$ domains, $0<\alpha<1$,   which are not NTA.
\end{thrm}

The class of ADP (admissible for the Dirichlet problem) domains was introduced in \cite{CGN1},
\cite{CGN2}, \cite{CGN3}, in connection with the study of the Dirichlet problem in the linear case $p=2$. We now recall the relevant definition.

\begin{dfn}\label{D:adp}
We say that a bounded domain is ADP if it is NTA and it satisfies the uniform outer ball condition.
\end{dfn}  

In the Euclidean setting every $C^{1,1}$ or convex domain is an example of an ADP domain. Thanks to Theorem \ref{T:MM} we have the following result.

\begin{prop}\label{P:adp}
In the Heisenberg group $\Hn$ every (Euclidean) $C^{1,1}$ domain which also satisfies the uniform outer ball condition is ADP.
\end{prop}

For instance, every (Euclidean) convex and $C^{1,1}$ domain is ADP. This follows from the fact, proved in \cite{LU}, that Euclidean convexity implies the uniform outer ball condition. The same property holds, more in general, in every Carnot group of step two, see \cite{CGN2}.

\section{Proof of Theorems \ref{T:eaADP} and \ref{T:pimproved}}\label{S:bb}

In this section we provide the proof of Theorems \ref{T:eaADP} and \ref{T:pimproved}. We begin by establishing a result which shows that a nonnegative $p$-harmonic function which vanishes near a boundary point at which there exists a gauge ball tangent from the outside, must vanish at most at a rate which is linear with respect to the  distance. In the linear case $p=2$ predecessors of this result were obtained in \cite{LU},
\cite{CGN2}, \cite{CGN3}. Hereafter, given an open set $\Om\subset\Hn$   we will use the notation
$$\Delta(g_0, R)=\partial\Om\cap B(g_0,R)$$
for a surface ball centered at $g_0\in\partial\Om$ with radius $R>0$.

\begin{proof}[Proof of Theorem \ref{T:eaADP}]
We will consider only the case $1<p<Q$ as the case $p=Q$ can be proceeded similarly.
Let $0<r<R_0/6$, where  $R_0$ is the smallest among the  $R_0$'s appeared in Definitions \ref{D:uob} and \ref{D:NTA}.
By the uniform Harnack inequality in \cite{CDG1} we know that $u>0$
in $\Om\cap B(g_0, 6r)$. Again by the uniform Harnack inequality we conclude that
\[ u(A_{5r}(g_0))\ \leq \ C\ u(A_r(g_0))\ , \] for some $C>0$
depending only on $n$ and $p$. By the Carleson estimate in
\cite{AS} which is valid for uniform domains, and hence for NTA
domains, in Carnot groups, we obtain for a constant $C =
C(n,\Om,p)>0$ such that
\begin{equation}\label{ce} u(g) \leq C u(A_{5r}(g_0)) \leq C u(A_{r}(g_0)) \ \text{for every}\ g\in \Om \cap B(g_0,5r).
\end{equation}
Fix now a point $g\in \Om\cap B(g_0,r)$ and let $\overline g\in \p
\Om$ be such that $d(g,\overline g) = d(g,\p \Om)$. Without loss of
generality we can assume that $d(g,\p \Om)\leq \frac{r}{2}$, otherwise the
conclusion follows immediately from \eqref{ce}. By the assumptions
on $\Om$ there exists an outer ball $B(g_1,\frac{r}{2})$ tangent to $\p \Om$
at $\overline g$. Since $d(g,\p \Om)\leq \frac{r}{2}$,  by the triangle inequality
\[ d(g_1,g_0)  \leq d(g_1,\overline g) + d(\overline g,g_0) \leq \frac{r}{2} + d(\overline g,g) +
d(g,g_0)\leq 2 r. \]
This implies that
\begin{equation}\label{i}
\Om \cap B(g_0,r) \subset \Om \cap B(g_1,3r)\subset \Om \cap
B(g_0,5r).
\end{equation}
By this inclusion and \eqref{ce} we infer that
\begin{equation*}
\frac{u(g')}{u(A_r(g_0))} \leq C,\ \ \text{for every}\ g'\in \Om
\cap B(g_1,3r).
\end{equation*}
On the other hand, if we consider the function
\[
f(g') = \frac{r^{(p-Q)/(p-1)} -
d(g',g_1)^{(p-Q)/(p-1)}}{r^{(p-Q)/(p-1)} - (3r)^{(p-Q)/(p-1)}} =
\frac{d(\overline g,g_1)^{(p-Q)/(p-1)} -
d(g',g_1)^{(p-Q)/(p-1)}}{r^{(p-Q)/(p-1)} - (3r)^{(p-Q)/(p-1)}} \ ,
\]
then by Theorem \ref{T:pfs} $f$ is $p$-harmonic in
$\Hn\setminus\{g_1\}$, $f\geq 0$ in $\overline \Om$, and $f\equiv 1$
on $\p B(g_1,3r) \cap \overline \Om$. Since by assumption $u$ vanishes continuously on $\p \Om \cap B(g_1,3r)$, by \eqref{ce} and the
comparison principle (Theorem \ref{T:ct}) we conclude that
\begin{equation}\label{ce3}
\frac{u(g')}{u(A_r(g_0))} \leq C f(g'),\ \ \text{for every}\
g'\in \Om \cap B(g_1,3r).
\end{equation}
From \eqref{ce3} and \eqref{i} we find in particular
\[
\frac{u(g)}{u(A_r(g_0))} \leq C f(g), \ \ \text{for every}\
g\in \Om \cap B(g_0,r).
\]
To complete the proof it will thus suffice to show that
\[
f(g) \leq C \frac{d(g,\overline g)}{r}.
\]
Applying the mean value theorem to the function $h(s) =
s^{(p-Q)/(p-1)}$, we find
\[
f(g) \leq C(Q,p)\ \frac{|d(g,g_1) - d(\overline g,g_1)|}{r} \leq
C(Q,p) \frac{d(g,\overline g)}{r}.
\]
This yields the desired conclusion.

\end{proof}

Next we prove a result which will be needed in the proof of Theorem \ref{T:pimproved}.

\begin{prop}\label{P:ea}
Let $\Om\subset \Hn$ be a bounded domain,
and let $1<p<\infty$. If for a given $g_0\in \p \Om$ there exists an
outer ball $B(g_1,r)\subset \Hn\setminus \overline \Om$ such that
$g_0\in \p B(g_1,r)$, then there exists $C>0$, depending only on
$n$ and $p$, such that if $\phi\in W^{1,p}_H(\Om) \cap C(\overline \Om)$,
$\phi \equiv 0$ on $\Delta(g_0,2r)$, then  for every $g\in \Om$ one
has
\[
|H^{\Om,p}_\phi(g)| \leq C \frac{d(g,g_0)}{r} \underset{\p
\Om}{max} |\phi|.
\]
\end{prop}

\begin{proof}
We can assume that $\underset{\p \Om}{max}\ |\phi|
>0$, otherwise there is nothing to prove. Since if $u$ is a weak solution of \eqref{pharm}, then for $\lambda >0$ one has $\mathcal L_p(\lambda u) = \lambda^{p-1} \mathcal L_p u = 0$, and so also $\lambda u$ is a weak solution, by considering $\psi = \phi/\underset{\p \Om}{max}\ |\phi|$, we
can also assume that $\underset{\p \Om}{max}\ |\phi| = 1$. By the
comparison principle (Theorem \ref{T:ct}) we obtain $|H^{\Om,p}_\phi|\leq 1$ in $\Om$. We
only discuss the case $1<p<Q$, since the cases $p=Q$ and $p>Q$ can
be treated in a completely analogous fashion. Consider the function
\[
f(g) = \frac{r^{(p-Q)/(p-1)} -
d(g,g_1)^{(p-Q)/(p-1)}}{r^{(p-Q)/(p-1)} - (2r)^{(p-Q)/(p-1)}},\ \
\ g\in \Om.
\]
Clearly, $f\geq 0$ in $\overline \Om$, $f(g_0) = 0$, $f\equiv 1$ on $\overline \Om
\cap \p B(g_1,2r)$, whereas $f\geq 1$ in $\overline \Om \cap
B(g_1,2r)^c$. Thanks to Theorem \ref{T:pfs}, $f$ is $p$-harmonic in
$\Om$. By Theorem \ref{T:ct} we obtain $|H^{\Om,p}_\phi|\leq f$ in
$\Om$. To finish the proof it thus suffices to show that
\[
f(g) \leq C  \frac{d(g,g_0)}{r},\ \ \ g \in \Om.
\]
Since $g_0\in \p  B(g_1,r)$ we have for every $g\in \Om$
\[
f(g) = \frac{d(g_0,g_1)^{(p-Q)/(p-1)} -
d(g,g_1)^{(p-Q)/(p-1)}}{r^{(p-Q)/(p-1)} - (2r)^{(p-Q)/(p-1)}}.
\]
From this observation, the sought for conclusion follows from a
standard application of the mean value theorem to the function $h(t)
= t^{(p-Q)/(p-1)}$ once we keep into account that $d(g,g_1)\geq r$
for $g\in \Om$.

\end{proof}

We can now present the proof of Theorem \ref{T:pimproved}.

\begin{proof}[Proof of Theorem \ref{T:pimproved}]
Once again, we only discuss the case $1<p<Q$, leaving
the details of the  case $p=Q$ to the interested reader. Let
$\Om\subset \Hn$ be a bounded domain satisfying the uniform outer
ball condition, and fix $g, g' \in \Om$. Let $\Gp(\cdot,g)$ be its
Green function with singularity at $g$. We begin by proving (i). If
either
\[
d(g',\p \Om) \geq \frac{d(g,g')}{4},\ \ \text{or}\ \ d(g',\p \Om)
\geq R_0,
\]
then from \eqref{gg} we obtain for some constant $C^*(n,\Om,p)>0$
\begin{equation*}
\Gp(g',g) \leq C^* d(g',\p \Om) d(g,g')^{(1-Q)/(p-1)},
\end{equation*}
and we are done.
We thus assume \[ d(g',\p \Om) < \frac{d(g,g')}{4},\ \
\text{and}\ \ d(g',\p \Om) < R_0, \] and set \[
r =
min\left(\frac{d(g,g')}{8},\frac{R_0}{2}\right).
\]
Notice that
$d(g',\p \Om) < 2r$.
Let $g_0\in \partial\Om$ be such that $d(g',\p \Om) = d(g',g_0)$. By the
assumption that we have made on $\Om$ there exists a ball
$B(g_1,r)\subset \Om^c$ such that $g_0\in \p B(g_1,r)$. We consider
the bounded open set $\Om_r = \Om \cap B(g_1,4r)$ and pick a
function $\phi\in W^{1,p}_H(\Om_r)\cap C(\overline{\Om_r})$, such
that $0\leq \phi \leq 1$, $\phi = 1$ on $\p B(g_1,4r) \cap \Om$ and
$\phi = 0$ on $\p \Om \cap B(g_1,2r)$. Let $H^{\Om_r,p}_\phi$ be the
solution to the Dirichlet problem for \eqref{pharm} with boundary
datum $\phi$. Since $\Om_r$ has a  outer tangent ball at $g_0$,
by Proposition \ref{P:ea} we obtain for every $g''\in \Om_r$
\[
|H^{\Om_r,p}_\phi(g'')| \leq C \frac{d(g'',g_0)}{r}.
\]

We now notice that the point $g'$ belongs to $\Om_r$. One has in
fact \[ d(g',g_1) \leq d(g',g_0) + d(g_0,g_1) = d(g',\p \Om) + r <
2r + r = 3r.
\]

We thus have
\begin{equation}\label{step2}
|H^{\Om_r,p}_\phi(g')| \leq C \frac{d(g',g_0)}{r} = C
\frac{d(g',\p \Om)}{r}.
\end{equation}

From the triangle inequality and $d(g',g_1) < 3r$ we find
\[
d(g,g_1) \geq d(g,g') - 3r \geq d(g,g') - \frac{3}{8} d(g,g') =
\frac{5}{8} d(g,g') \geq 5r,
\]
or equivalently
\[
g \in \Hn \setminus B(g_1,5r).
\]

For $g^*\in \Om_r$ we now define
\[
w(g^*) = C^{-1}\left(\frac{d(g,g')}{8}\right)^{(Q-p)/(p-1)}
\Gp(g^*,g),
\]
where $C>0$ is the constant in \eqref{gg}. Since $g\not\in \Om_r$, from Theorem \ref{T:pfs} we
see that $w$ is $p$-harmonic in $\Om_r$. Moreover, when $g^* \in \p
B(g_1,4r)$ we have \[ d(g^*,g) \geq d(g,g_1) - d(g^*,g_1) \geq
5r - 4r \geq  \frac{d(g,g')}{8}.
\]

We thus obtain in view of \eqref{gg} for any $g^*\in \p \Om_r \cap
\Om$
\[
w(g^*) \leq C^{-1}\left(\frac{d(g,g')}{8}\right)^{(Q-p)/(p-1)} C
d(g^*,g)^{(p-Q)/(p-1)} \leq 1.
\]

On the other hand $w = 0$ on $\Om_r \cap \p \Om$. If instead we look
at $H^{\Om_r,p}_\phi$, then we have $H^{\Om_r,p}_\phi = 1$ on $\p
\Om_r \cap \Om$, whereas $H^{\Om_r,p}_\phi \geq 0$ on $\Om_r \cap \p
\Om$. By Theorem \ref{T:ct} we conclude that  $w\leq
H^{\Om_r,p}_\phi $ in $\Om_r$. In particular, we must have $w(g')
\leq H^{\Om_r,p}_\phi(g')$. Combining this with \eqref{step2} we
finally obtain
\[
C^{-1}\left(\frac{d(g,g')}{8}\right)^{(Q-p)/(p-1)} \Gp(g',g) \leq C
\frac{d(g',\p \Om)}{r}.
\]

To reach the desired conclusion it now suffices to observe that
$$\frac{d(g,g')}{r} \leq max\left(8,\frac{2 diam(\Om)}{R_0}\right) =
C(\Om).$$

This proves part (i) of the theorem.
Next, we prove part (ii). Suppose that $\Gp$ is symmetric, i.e.,
$\Gp(g',g) = \Gp(g,g')$. But then from part (i) we obtain for every
$g,g'\in \Om$
\begin{equation}\label{step3}
\Gp(g',g) = \Gp(g,g') \leq C \frac{d(g,\p
\Om)}{d(g,g')^{(Q-1)/(p-1)}}.
\end{equation}

We now argue exactly as in the proof of part (i) except that we
define \[ w(g^*) = C^{-1} d(g,\p \Om)^{-1} d(g,g')^{(Q-1)/(p-1)}
\Gp(g^*,g),\ \ g^*\in \Om_r.
\]

Using \eqref{step3} instead of \eqref{gg} we reach the desired
conclusion.

The proofs of (iii) and (iv) are left to the reader.

\end{proof}

\section{Non-characteristic segments}\label{S:qs1}

The remainder of this paper is devoted to proving Theorems \ref{T:eb} and \ref{T:cp}.  Before we can do so, however, we need to develop some
preliminary delicate analysis aimed at constructing, away from the characteristic set of any $C^{1,1}$ domain $\Om\subset \Hn$, a suitable family of
paths connecting a given non-characteristic point $\overline g\in
\p \Om$ to a point $g(\lambda)\in  \Om$ which is the center of an
interior tangent gauge ball at $\overline g$. Among the important
features of these paths are: 1) The fact that for every $\lambda$
the point $\overline g$ realizes the distance of $g(\lambda)$ to
$\p \Om$; 2) A quasi-segment property with respect to the gauge
distance holds along the path itself. These paths will play a
crucial role in the proof of Theorems \ref{T:eb}, \ref{T:cp}.

In what follows, we will use the notation $z = (x,y)$, $z_0 = (x_0,y_0)$ for points of $\R^{2n}$. Given a vector $\omega = (a,b)\in \R^{2n}\setminus \{0\}$, we will indicate with
\[
H_\omega = \{(z,t)\in \Hn\mid <z,\omega> = <a,x> + <b,y>  >0\},
\]
the vertical half-space whose boundary will be denoted by
\[
\Pi_\omega =  \{(z,t)\in \Hn\mid <z,\omega>  = 0\}.
\]
Without restriction we will assume throughout this section that 
\[
|\omega|^2 = |a|^2 + |b|^2 = 1.
\]
We observe explicitly that the vertical hyperplane $\Pi_\omega$ has empty characteristic set.
 
\begin{lemma}\label{L:vp}
Consider the vertical half-space $H_{\omega}$.
Given a point $g = (z_0,t_0)\in H_\omega$,  one has
\[
d(g,\Pi_\omega) = d(g,\overline g)  = <z_0,\omega>,
\]
where  $\overline g = (\overline z,\overline t)$, with
\begin{equation}\label{dvp} \overline z = z_0 - <z_0,\omega> \omega, \ \ \ \overline t = t_0 + \frac{1}{2} <z_0,\omega><z_0^\perp,\omega>. \end{equation}
Furthermore, one has
\begin{equation}\label{dv5}
B(g,<z_0,\omega>) \subset H_\omega,\ \ \text{and}\ \ \overline g \in \p
B(g,<z_0,\omega>) \cap \Pi_\omega.
\end{equation}
Finally, the straight half-line in $\R^{2n+1}$ originating at $\overline g$ and
parallel to the vector $g - \overline g$
 \begin{equation}\label{dv7}
 g(\lambda) = \left(
 z_0 + (\lambda -1)<z_0,\omega>\omega,t_0 - \frac{\lambda-1}{2}<z_0, \omega><z_0^\perp,\omega>\right),\ \ \lambda \ge 0,
 \end{equation}
possesses the property of being a \emph{segment} with respect to the
gauge distance. By this we mean that for every $\lambda \geq 1$ we have
\begin{equation}\label{dv8}
d(g(\lambda),\overline g) = d(g(\lambda),g) + d(g,\overline g).
\end{equation}

\end{lemma}

\begin{proof}
Given a point $g = (z_0,t_0)\in H_\omega$ consider the fourth
power of the gauge distance of $g$ from a generic point
$(z,t)\in \Pi_\omega$. Denoting by $f(z,t)$ such function, we obtain from \eqref{hgl2}
\begin{equation*}
f(z,t)  =  d((z_0,t_0), (z,t))^4 
= |z-z_0|^4 + 16 \left(t-t_0 +\frac{1}{2}<z_0^\perp,z>\right)^2.
\end{equation*}
Since we want to minimize $f$ subject to the constraint that $(z,t)\in \Pi_\omega$, by the method of Lagrange multipliers we see that the critical points of $f$, subject to the constraint $<z,\omega> = 0$, are the solutions of
\begin{equation}\label{dv2}
\begin{cases}
4 |z-z_0|^2(z-z_0) = \lambda \omega,
\\
t - t_0 + \frac{1}{2}<z_0^\perp,z> = 0.
\end{cases}
\end{equation}
Taking the inner product of the first equation in \eqref{dv2} with $\omega$, we easily recognize that it must be
\[
\lambda = - 4 |z-z_0|^2 <z_0,\omega>.
\]
The value $\lambda = 0$ must be discarded, as in view of \eqref{dv2} it gives $z = z_0$, $t = t_0$, and therefore we would conclude that $f$ has its minimum value $(=0)$ at such point. Now, for $\lambda \not= 0$ we obtain $z\not= z_0$ from \eqref{dv2}, and therefore we conclude that the point  $\overline g = (\overline z, \overline t) \in \Pi_\omega$, at which $f$ attains its maximum value, has coordinates 
\[
\overline z = z_0 - <z_0,\omega> \omega,\ \ \ \ \overline t = t_0 + \frac{1}{2} <z_0,\omega><z_0^\perp,\omega>,
\]
which proves \eqref{dvp}. With this information in hands, a simple computation shows that
\[
d(g,\overline g) = d(g,\Pi_\omega) = <z_0,\omega>.
\]
From the latter equation and from \eqref{dvp} it follows immediately that
\begin{equation*}
\overline g \in \p B(g,<z_0,\omega>) \cap \Pi_\omega,
\end{equation*}
thus proving the second part of \eqref{dv5}.
For the first part of \eqref{dv5}, we need to show that
\[
 B(g,<z_0,\omega>) \subset
H_\omega.
\]
To see this inclusion it suffices to observe that: 1) Every gauge ball
is convex (in the Euclidean sense); 2) The manifold $\p B(g,<z_0,\omega>)$ is
tangent to the vertical plane $ \Pi_\omega$ at $\overline g$. Now 1)
follows from the fact that any gauge ball centered at the origin is
obviously (Euclidean) convex, and the left-translations generated by \eqref{hgl}, being affine maps,
preserve convex sets. To prove 2) it suffices to show that the
Euclidean unit normal to $\p B(g,<z_0,\omega>)$ at $\overline g$ is
parallel to $\omega\in \mathbb{R}^{2n+1}$. Now with
$$F(z,t)=d((z_0,t_0), (z,t))^4 - <z_0,\omega>^4,$$
a computation using \eqref{dvp} shows that such a Euclidean normal
is given by
\[
\nabla F(\overline g)  = - 4<z_0,\omega>^3 \omega,
\]
and so we are done.

Finally, we want to show that $g(\lambda)$
defined in \eqref{dv7} possesses the segment property \eqref{dv8}
with respect to the gauge distance. This is equivalent to showing that for any $\lambda \ge 1$ one has
\begin{equation}\label{dv10}
d(g(\lambda),\overline g) - d(g(\lambda),g) = d(g,\overline g).
\end{equation}
Using \eqref{hgl2} we have for any $\lambda \ge 1$,
\[
d(g(\lambda),g) = (\lambda - 1)<z_0,\omega> = (\lambda - 1) d(\overline g,g).
\]
By a similar computation we find
\[
d(g(\lambda),\overline g) = \lambda <z_0,\omega> = \lambda  d(\overline g,g).
\]
The desired conclusion \eqref{dv10} thus follows. 

\end{proof}

\begin{figure}[h]
\begin{center}
\includegraphics[width=50mm]{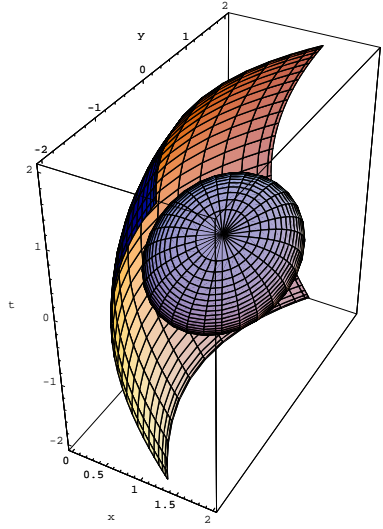}
\caption{The tangent ball at a point of the vertical plane $x=0$ in $\mathbb{H}^1$.}
\end{center}
\end{figure}

\begin{rmrk}\label{R:hp}
We stress that given $g_0\in H_\omega$, the point $\overline g_0$ in
\eqref{dvp} which realizes the gauge distance of $g_0$ to the
boundary $\Pi_\omega$ belongs to the \emph{horizontal plane}
through $g_0$. We recall that the horizontal plane through a point
$g_0 = (z_0,t_0)\in \Hn$ is given by
\[
\left\{(z,t) \in \Hn\mid t = t_0 - \frac{1}{2} <z_0^{\perp},z>\right\},
\]
see for instance \cite{DGN}. We also notice that from the proof of Lemma \ref{L:vp} it follows that
\[
d(g,\Pi_\omega) = d_e(g,\Pi_\omega),
\]
where we have indicated with $d_e(g,\Pi_\omega)$ the Euclidean distance in $\R^{2n+1}$ from $g$ to the vertical hyperplane $\Pi_\omega$.
\end{rmrk}

\begin{thrm}\label{L:vptb}
Let $\Om\subset \Hn$ be a (Euclidean) $C^{1,1}$ domain. Suppose
that at a given point $\overline g\in \p \Om$ the tangent
hyperplane to $\p\Om$  is the vertical hyperplane $\Pi_\omega$, and that $<\nu(\overline g), \omega>=1$, where $\nu(\overline g)$ is the unit inward normal
to $\p \Om$ at $\overline g$. This means, in particular, that $\overline g =
(\overline z,\overline t)$, with $<\overline z,\omega>=0$. Consider the straight
half-line segment whose points are given by
\begin{equation*}
g(\lambda) = \left(\overline z + \lambda \omega,\overline t -
\frac{\lambda}{2} <\overline z^\perp,\omega>\right),\ \ \ \ \lambda > 0.
\end{equation*}
Then for every $\lambda>0$ the gauge ball $B(g(\lambda),\lambda)$ is tangent to $\Pi_\omega$
at $\overline g$, and there exists $\lambda_0>0$ depending only on the
$C^{1,1}$ character of $\Om$ such that for every
$0<\lambda<\lambda_0$ one has
\begin{equation}\label{cont}
B(g(\lambda),\lambda) \subset \ \Om.
\end{equation}
\end{thrm}

\begin{proof}
Using \eqref{dvp} it is not difficult to verify that for any given
$\lambda>0$ we have
\[
d(g(\lambda), \Pi_\omega) = d(g(\lambda),\overline g) = \lambda.
\]
As we have seen in the proof of Lemma \ref{L:vp}, this shows that the gauge ball $B(g(\lambda),\lambda)$ is fully contained in the vertical half-space $H_\omega$, and tangent to $\Pi_\omega$ at $\overline g$. To prove \eqref{cont} we argue as follows. We let $S\in U(n)$ be a unitary matrix such that $S\omega = e_1 = (1,0,...,0)$. Such transformation sends the vertical hyperplane $\Pi_\omega = \Pi(\overline g)$ into the 
vertical hyperplane 
\[
\Pi_1=\{(x,y,t)\in \Hn\mid x_1=0\}.
\]
Furthermore, thanks to Lemma \ref{L:inv} such transformation preserves the gauge distance, see \eqref{inv}, and therefore it is not restrictive to assume from the start that 
the hyperplane $\Pi(\overline g)$ coincides with $\Pi_1.$ And that the point $\overline g$ at which the domain $\Om$ and the hyperplane $\Pi_1$ touch is a non-characteristic point for $\p \Om$.

Having done this, the boundary of the gauge ball $B(g(\lambda),\lambda)$ is now described
by the equation
\begin{align}\label{bvgb}
&\bigg((x_1-\lambda)^2 + |x'-\overline x'|^2 + |y - \overline
y|^2\bigg)^2 \\
& + \ 16 \left(t - \overline t - \frac{\lambda}{2}(y_1 - \overline
y_1) + \frac{1}{2}\left(<x,\overline y> - <\overline
x',y'>\right)\right)^2\ =\ \lambda^4\ ,\notag
\end{align}
where we have let $(x,y)= (x_1, x', y_1, y')$ and $(\overline x,\overline y)= (\overline x_1, \overline x', \overline y_1, \overline y') \in \R\times \R^{n-1}\times \R\times \R^{n-1}$.

Since we are assuming that $\p \Om$ is tangent to the vertical plane
$x_1 = 0$ at $\overline g$, we can locally describe the boundary of
$\Om$ as a graph in the variables $(x',y,t)$. This means that we can
find $r_0>0$ sufficiently small, and a $C^{1,1}$ function

$$\phi:\left\{(x',y,t)\in \R^{n-1}\times \Rn\times \R \mid |x'-\overline x'|^2
+ |y-\overline y|^2 + (t - \overline t)^2 <r_0^2\right\} \to \R ,$$
such that $\phi(\overline x',\overline y, \overline t) = 0$, $\nabla
\phi(\overline x',\overline y,\overline t) = 0$, and for which  the set
$$\p \Om \cap \left \{(x,y,t)\in \Hn\mid
|x'-\overline x'|^2 + |y-\overline y|^2 + (t - \overline t)^2
<r_0^2, \ |x_1|<  r_0^2\right \}$$
is given by
$$\left \{(x,y,t)\in \Hn \mid x_1 = \phi(x',y,t),\ |x'-\overline x'|^2 + |y-\overline y|^2 + (t - \overline t)^2
<r_0^2, \ |x_1|<  r_0^2\right\},$$
whereas,  since  $<\nu(\overline g), e_1>=1>0$, the set
$$ \Om \cap \left \{(x,y,t)\in \Hn\mid
|x'-\overline x'|^2 + |y-\overline y|^2 + (t - \overline t)^2
<r_0^2, \ |x_1|<  r_0^2\right \}$$
is given by
$$\left \{(x,y,t)\in \Hn \mid x_1 > \phi(x',y,t),\ |x'-\overline x'|^2 + |y-\overline y|^2 + (t - \overline t)^2
<r_0^2, \ |x_1|<  r_0^2\right \}.$$

By the $C^{1,1}$ assumption on $\Om$ we can find $A>0$ such  that
\begin{equation}\label{parver}
|\phi(x', y,t)|\ \leq\ A\left(|x'-\overline x'|^2 + |y-\overline y|^2 + (t -
\overline t)^2\right),
\end{equation}
whenever
$$|x'-\overline x'|^2 + |y-\overline y|^2 +
(t - \overline t)^2 <r_0^2. $$

To prove \eqref{cont} it will thus suffice to show that the gauge ball
$B(g(\lambda),\lambda)$ is entirely contained in the paraboloid with
respect to the variables $(x',y,t)$ in the right-hand side of
\eqref{parver}.

To simplify the situation we left-translate $\Om$ and $\Pi_1$ by the
point $g_0 = (0,- \overline x',-\overline y,-\overline t)\in \Pi_1$.
Such a left-translation leaves $\Pi_1$ unchanged, but has the effect
that now the boundary of the gauge ball \eqref{bvgb}  becomes
\begin{equation}\label{bvgb2}
\bigg((x_1-\lambda)^2 + |x'|^2 + |y|^2\bigg)^2 + 16 \left(t  -
\frac{\lambda}{2} y_1\right)^2 = \lambda^4,
\end{equation}
whereas the paraboloid in the right-hand side of \eqref{parver} is
now given by
\begin{equation}\label{parver2}
x_1 = A (|x'|^2 + |y|^2 + t^2),\ \ \ \ |x'|^2 + |y|^2 + t^2 <
r^2_0.
\end{equation}

The advantage is that we can now easily solve with respect to the
variable $x_1$ the equation of order four \eqref{bvgb2} obtaining
\begin{equation}\label{bvgb3}
x_1 = \lambda\ - \sqrt{\sqrt{\lambda^4 - 16 \left(t -
\frac{\lambda}{2} y\right)^2} - (|x'|^2 + |y|^2)},
\end{equation}

Notice that the variable $x_1$ in \eqref{bvgb2} ranges from $0$ to
$2\lambda$, and that the allowable region of points $(x',y,t)$ is
obtained by projecting onto the $(x',y,t)$-hyperplane the
intersection of \eqref{bvgb2} with the plane $x_1 = \lambda$, which
gives
\begin{equation}\label{ar}
\left(|x'|^2 + |y|^2\right)^2 + 16 \left(t  - \frac{\lambda}{2}
y_1\right)^2\ <\ \lambda^4.
\end{equation}

To further simplify the situation we consider the global
diffeomorphism of $\R^{2n+1}$ onto itself given by
\[
\xi = x,\ \ \eta = y,\ \ \ \ \tau = t - \frac{\lambda}{2}
y_1.
\]

Such diffeomorphism transforms \eqref{bvgb3} and \eqref{parver2}
respectively into
\begin{equation*}
\xi_1 = \lambda\ - \sqrt{\sqrt{\lambda^4 - 16 \tau^2} - (|\xi'|^2
+ |\eta|^2)},
\end{equation*}
and
\begin{equation*}
\xi_1 = A \left(|\xi'|^2 + |\eta|^2 + \left(\tau +
\frac{\lambda}{2} \eta_1\right)^2\right),\ \ \ \ |\xi'|^2 +
|\eta|^2 + \left(\tau + \frac{\lambda}{2} \eta_1\right)^2 < r^2_0,
\end{equation*}
and the region \eqref{ar} into
\begin{equation}\label{ar2}
\left(|\xi'|^2 + |\eta|^2\right)^2 + 16 \tau^2 <  \lambda^4.
\end{equation}

Notice that \eqref{ar2} imposes that $\lambda^4 - 16 \tau^2 >
\left(|\xi'|^2 + |\eta|^2\right)^2 \geq 0$. After these reductions
we are left with proving that if $\lambda$ is sufficiently small
then
\begin{equation}\label{trans}
\lambda - \sqrt{\sqrt{\lambda^4 - 16 \tau^2} - (|\xi'|^2 +
|\eta|^2)}
> A \left(|\xi'|^2 + |\eta|^2 + \left(\tau + \frac{\lambda}{2}
\eta_1\right)^2\right),
\end{equation}
provided that \eqref{ar2} holds. We find
\begin{eqnarray*}
\lambda - \sqrt{\sqrt{\lambda^4 - 16 \tau^2} - (|\xi'|^2 +
|\eta|^2)}&=&
\frac{\lambda^2 - \sqrt{\lambda^4 - 16 \tau^2} + (|\xi'|^2 + |\eta|^2)}{\lambda + \sqrt{\sqrt{\lambda^4 - 16 \tau^2} - (|\xi'|^2 + |\eta|^2)}}\\
&\geq& \frac{\lambda^2 - \sqrt{\lambda^4 - 16 \tau^2} + (|\xi'|^2 +
|\eta|^2)}{2\lambda }.
\end{eqnarray*}

Thus for \eqref{trans} to hold it is enough to find $\lambda_0>0$ so
that for $0<\lambda<\lambda_0$,
\begin{equation*}
\lambda^2 - \sqrt{\lambda^4 - 16 \tau^2} + (|\xi'|^2 + |\eta|^2)
>  2 A \lambda \left(|\xi'|^2 + |\eta|^2 + \left(\tau +
\frac{\lambda}{2} \eta_1\right)^2\right),
\end{equation*}
provided that \eqref{ar2} holds. Note that if $0<\lambda<2$, then
\begin{eqnarray*}
\lambda^2 - \sqrt{\lambda^4 - 16 \tau^2} + |\xi'|^2 + |\eta|^2
&=& \frac{ 16 \tau^2}{\lambda^2 + \sqrt{\lambda^4 - 16 \tau^2}} + |\xi'|^2 + |\eta|^2 \\
&\geq& \frac{ 8 \tau^2}{\lambda^2} + |\xi'|^2 + |\eta|^2 \geq 2
\tau^2 + |\xi'|^2 + |\eta|^2.
\end{eqnarray*}

On the other hand, we easily have \[ 2 A \lambda \left(|\xi'|^2 +
|\eta|^2 + \left(\tau + \frac{\lambda}{2} \eta_1\right)^2\right)
\leq 6 A \lambda (|\xi'|^2 + |\eta|^2 + \tau^2).\]

It thus suffices to show that
\[
2 \tau^2 + |\xi'|^2 + |\eta|^2 \ >\ 6 A \lambda (|\xi'|^2 + |\eta|^2
+ \tau^2) \ .
\]

It is now easy to verify that this latter inequality is valid
provided that $0<\lambda< 1/6A$. We conclude that \eqref{trans}
holds for $0<\lambda<\lambda_0 = \min\{2,1/6A\}$.

\end{proof}

\section{Characteristic quasi-segments}\label{S:qs2}

In this section we study the distance from a characteristic hyperplane away from the
characteristic set.

\begin{lemma}\label{L:dcp}
Consider the half-space $H_0 = \{(z,t)\in \Hn\mid t>0\}$ whose
boundary is the characteristic hyperplane $\Pi_0 = \{(z,0)\in
\Hn\mid z\in \R^{2n}\}$. For any point $(z_0,t_0)\in H_0$, with
$z_0\not= 0$, its gauge distance to $\Pi_0$ is realized by the point $\overline g_0 =(z_0+\lambda z_0^{\perp},0)\in H_0$ and is
given by the formula
\begin{equation}\label{dcp}
d((z_0,t_0),\Pi_0) = \left(\lambda^4 |z_0|^4 + 16\left(t_0 -
\frac{\lambda}{2}|z_0|^2\right)^2\right)^{1/4},
\end{equation}
where $\lambda = \lambda(z_0,t_0)$ is the real root of the cubic equation $\lambda^3 + 2 \lambda \ =\ \frac{4t_0}{|z_0|^2}$.
Equivalently, $\lambda= G\left(\frac{2t_0}{|z_0|^2}\right) >0$ with $G$ being given by the equation
\eqref{G} below. Moreover, one has
\begin{equation}\label{las3}
d((z_0,t_0), \Pi_0) = \frac{2t_0}{|z_0|} (1 + o(1)),\ \ \
\text{as}\ t_0 \to 0^+,
\end{equation}
where $o(1)$ indicates a function which goes to zero as
$t_0/|z_0|\to 0$. Keeping in mind that $d_e((z_0,t_0),H_0) = t_0$,
this gives in particular
\begin{equation*}
d((z_0,t_0), H_0) = \frac{2 d_e((z_0,t_0),H_0) }{|z_0|} (1 +
o(1)),\ \ \ \text{as}\ t_0 \to 0^+.
\end{equation*}
\end{lemma}

\begin{proof}
Let $(z_0,t_0)\in H_0$ be such that $z_0 = (x_0,y_0) \not= 0$, and
consider the function \[ f(z) = d((z_0,t_0),(z,0))^4. \] 

From \eqref{hgl2} we find (recall that $z_0^{\perp}=(y_0,-x_0)$)
\begin{equation}\label{f}
f(z) =  |z-z_0|^4 + 16 \left(t_0 - \frac{1}{2}
<z,z_0^\perp>\right )^2.
\end{equation}

 The possible critical
points of $f$ are solutions to the equation
\[
\nabla f(z) = 4 |z-z_0|^2 (z-z_0) - 16 \left(t_0 - \frac{1}{2}
<z,z_0^\perp>\right) z_0^\perp = 0.
\]

Notice that $z=z_0$ cannot possibly be a critical point of $f$ since
$\nabla f(z_0) = - 16 t_0 z_0^\perp \not= 0$. This forces \[ t_0 -
\frac{1}{2} <z,z_0^\perp> \not= 0, \] at a critical point $z$
for otherwise we would have to have $z= z_0$.  Also notice that
$z\in \R^{2n}$ is a critical point of $f$ iff we have for $z$
\begin{equation}\label{cc}
 |z-z_0|^2 (z-z_0)  =  4 \left(t_0 - \frac{1}{2}
<z,z_0^\perp>\right) z_0^\perp.
\end{equation}

This means that $z$ must satisfy the equation
\begin{equation}\label{cp}
z   = z_0 + \lambda\ z_0^\perp,
\end{equation}
for the real number $\lambda$ given by
\begin{equation}\label{lambda}
\lambda = \frac{4 (t_0 -\frac{1}{2}
<z,z_0^\perp>)}{|z-z_0|^2}.
\end{equation}

From what we have observed above, it must be that $|\lambda|>0$. At
a critical point we have from \eqref{cp}
\[
<z,z_0^\perp> = <z_0 + \lambda z_0^\perp, z_0^\perp> =  \lambda
|z_0|^2,\ \ \ \ \ \  |z-z_0|^2 = \lambda^2 |z_0|^2.
\]

Substituting these equations and \eqref{cp} in \eqref{cc}, we find
\begin{equation*}
\lambda^3 |z_0|^2 z_0^\perp = 4 \left(t_0 - \frac{\lambda}{2}
|z_0|^2\right) z_0^\perp.
\end{equation*}

Taking the inner product of both sides with $z_0^\perp$ we obtain
\begin{equation}\label{cc3}
|z_0|^2 \lambda^3  = 4 \left(t_0 - \frac{\lambda}{2}
|z_0|^2\right).
\end{equation}

From \eqref{cc3} we conclude that $\lambda$ must satisfy the cubic
equation
\begin{equation}\label{cubic}
\lambda^3 + 2 \lambda  = \frac{4t_0}{|z_0|^2}.
\end{equation}

If we consider the strictly increasing function on $[0,\infty)$
\begin{equation}\label{Psi}
\Psi(\lambda) =
\frac{\lambda^3}{2} + \lambda,
\end{equation}
then \eqref{cubic} can be written
\begin{equation}\label{psi1}
\Psi(\lambda) = b,\ \ \ \text{with}\ b = \frac{2t_0}{|z_0|^2}.
\end{equation}

Let now $G= \Psi^{-1}:[0,\infty)\to \R$ be the inverse function of
$\Psi$, using the Cardano-Tartaglia formula, see \cite{Ca}, we find
for $b\geq 0$
\begin{equation}\label{G}
G(b) = \left( \left(\frac{8}{27} + b^2 \right)^{1/2} +
b\right)^{1/3} - \left( \left(\frac{8}{27} + b^2\right)^{1/2} -
b\right)^{1/3}.
\end{equation}
It is clear from \eqref{G} that $G(0) = 0$. We conclude that one
real root of \eqref{cubic} is given by
\begin{align}\label{ct}
\lambda & = \lambda(z_0,t_0) = G\left(\frac{2t_0}{|z_0|^2}\right)
\\
& = \left( \left(\frac{8}{27} + \frac{4 t_0^2}{|z_0|^4}
\right)^{1/2} + \frac{2t_0}{|z_0|^2}\right)^{1/3} - \left(
\left(\frac{8}{27} + \frac{4 t_0^2}{|z_0|^4}\right)^{1/2} -
\frac{2t_0}{|z_0|^2}\right)^{1/3} > 0. \notag
\end{align}
Notice that as $t_0\to 0^+$ one has $\lambda(z_0,t_0) \to 0$ for
every $z_0\not= 0$ fixed. We also notice that since $\Psi'(0) = 1$
and $\Psi''(0) = 0$, the inverse function theorem gives
\begin{equation}\label{infn} G'(0) = \frac{1}{\Psi'(0)} = 1,\
\ \ G''(0) = - \frac{\Psi''(0)}{\Psi'(0)^3} = 0.
\end{equation}
We observe that \eqref{psi1} gives
\begin{equation}\label{tpsi} t_0
= \frac{|z_0|^2}{2}\ \Psi(\lambda).
\end{equation}
Since $t_0>0$, the other two roots of \eqref{cubic} are necessarily
complex conjugates, and therefore they are to be discarded since
$\lambda\in \R$, see \eqref{lambda}. Equation \eqref{cp} thus
produces one single critical point $\overline z_0$ with $\lambda$
given by \eqref{ct}. From \eqref{f} we thus conclude that for
$z_0\not= 0$
\begin{equation}\label{fd}
d((z_0,t_0),\Pi_0) = f(\overline z_{0})^{1/4} = \left(\lambda^4
|z_0|^4 + 16\left(t_0 -
\frac{\lambda}{2}|z_0|^2\right)^2\right)^{1/4},
\end{equation}
which gives \eqref{dcp}. If we keep \eqref{cc3} in mind, we can re-write this formula as
follows
\begin{equation*}
d((z_0,t_0), \Pi_0) = |z_0| \lambda (1 + \lambda^2)^{1/4}.
\end{equation*}
Therefore,
\begin{equation}\label{lp}
\lambda = \psi^{-1}\left(\frac{d((z_0,t_0),\Pi_0)}{|z_0|}\right),\ \ \text{where}\ \ \psi(s) \overset{def}{=} s (1 + s^2)^{1/4}.
\end{equation}
We note that $\psi:[0,\infty)\to \R$ is strictly increasing and
that, since $\psi'(0) = 1, \psi''(0) = 0$, we have
\begin{equation*}
(\psi^{-1})'(0) = 1,\ \ \ \ (\psi^{-1})''(0) = 0.
\end{equation*}
Therefore
\begin{equation}\label{psinv2}
\psi^{-1}(s) = s (1 + O(s^2)),\ \ \text{as} \ s\to 0^+.
\end{equation}
This shows that
\begin{equation}\label{las}
\lambda(z_0,t_0) = \frac{d((z_0,t_0),\Pi_0)}{|z_0|} (1 + o(1)),\ \ \text{as}\ t_0 \to 0^+,
\end{equation}
with $o(1)\to 0$ as $t_0/|z_0|\to 0$.  On the other hand,
\eqref{ct} and \eqref{infn} imply that
\begin{equation}\label{las2}
\lambda(z_0,t_0) = \frac{2t_0}{|z_0|^2} (1 + o(1)).
\end{equation}
From \eqref{las}, \eqref{las2} we thus conclude
\[
d((z_0,t_0),\Pi_0) = \frac{2t_0}{|z_0|} (1 + o(1)),\ \ \
\text{as}\ t_0 \to 0^+,
\]
which proves \eqref{las3} and completes the proof of the lemma.

\end{proof}

\begin{rmrk}\label{R}
Before proceeding further we observe explicitly that, contrarily to what
happens in the case of a vertical plane, in the present situation
the point $\overline g_0$ which realizes the gauge distance of $g_0$
to $\Pi_0$ does not belong to the horizontal plane through $g_0$.
To see this let us
recall that, given a point $g_0=(z_0,t_0)$, then the equation of such plane is
given by
\begin{equation}\label{hp*}
t = t_0 - \frac{1}{2} <z_0^\perp,z>.
\end{equation}
Using $\overline g_0=(z_0+\lambda z_0^{\perp},0)$ (by Lemma \ref{L:dcp}) in \eqref{hp*} we find that  the condition $\overline g_0$ belongs to the horizontal plane through $g_0$ is equivalent to
saying that
\[
t_0 - \frac{\lambda}{2} |z_0|^2 = 0.
\]
But this is impossible by \eqref{cc3} and by the fact that
$\lambda= G\left(\frac{2t_0}{|z_0|^2}\right) >0$.
\end{rmrk}

The next result shows that the ball centered at $g_0\in H_0$ and
with radius $d(g_0,\overline g_0)$ is tangent to $\Pi_0$ at
$\overline g_0$.

\begin{lemma}\label{L:tangball}
Let $g_0 = (z_0,t_0)\in \Hn$ with $z_0\not= 0$ and $t_0>0$, then the
gauge ball $B(g_0,d(g_0,\overline g_0))$ is tangent to $\Pi_0$ at
$\overline g_0$, and
\begin{equation}\label{abc}
B(g_0,d(g_0,\overline g_0)) \subset H_0,\ \ \text{and}\ \ \overline
g_0 \in \Pi_0,
\end{equation}
where $\overline g_0$ is as in Lemma \ref{L:dcp}.
\end{lemma}

\begin{proof}
To prove this we set $R_0 = d(g_0,\overline g_0)$ and consider the
function
\begin{equation*}
F(z,t) =\  d((z_0, t_0), (z, t))^4  - R_0^4\  =\ |z-z_0|^4 + 16 \left(t-t_0 +\frac{1}{2} <z_0^\perp,z>\right)^2\
-\ R_0^4\ 
\end{equation*}
by formula \eqref{hgl2}. We have
\[
\nabla F(z,t)\ =\ 4 \left(|z-z_0|^2 (z-z_0) + 4 \left(t -t_0 +
\frac{<z_0^\perp,z>}{2} \right) z_0^\perp,\ 8 \left(t -t_0 +
\frac{<z_0^\perp,z>}{2} \right)\right) \  .
\]

Now with $\overline g_0\ = (\overline z_0, 0)\ =\ (z_0 + \lambda z_0^\perp,0)$,
we find that
\[
\nabla F(\overline g_0)\ =\ 4\left(\lambda^3 |z_0|^2 z_0^\perp - 4
\left(t_0 - \frac{\lambda}{2} |z_0|^2\right) z_0^\perp,- 8 \left(
t_0 - \frac{\lambda}{2} |z_0|^2\right)\right)\ .
\]

Using \eqref{cc3} we conclude that
\[
\nabla F(\overline g_0)\ =\ - 32 \left(0,\left( t_0 -
\frac{\lambda}{2} |z_0|^2\right)\right)\ =\ - 32 \frac{|z_0|^2}{2}
\left(\frac{2 t_0}{|z_0|^2} -
G\left(\frac{2t_0}{|z_0|^2}\right)\right) (0,1)\ .
\]

Thus the gauge ball $B(g_0,R_0)$ is tangent to $\Pi_0$ at $\overline g_0$.
Since $B(g_0,R_0)$ is convex, we see that it must obey the inclusion in \eqref{abc}.

\end{proof}

\begin{figure}[h]
\begin{center}
\includegraphics[width=82mm]{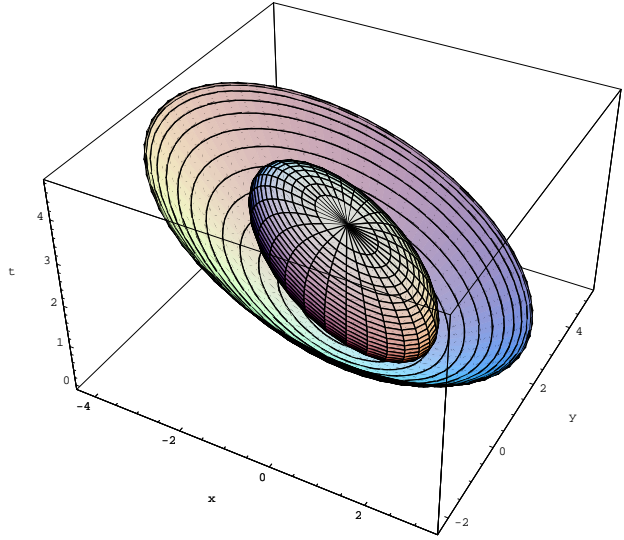}
\caption{The tangent ball at a non-characteristic point of the horizontal plane $t=0$.}
\end{center}
\end{figure}

Suppose now we are given a point $\overline g = (\overline z,0)\in
\Pi_0$. For any given $\lambda \in \R$ we want to find all solutions
$z$ of the equation
\begin{equation}\label{curves}
z + \lambda z^\perp = \overline z.
\end{equation}
The matrix of this system is
\[
A_\lambda = \begin{pmatrix} 1 & \lambda
\\
- \lambda & 1
\end{pmatrix}
\]
whose determinant is $1+\lambda^2>0$, and thus $A_\lambda$ is
invertible, and one has
\[
A_\lambda^{-1}\ =\ \begin{pmatrix} \frac{1}{1+\lambda^2} &
 - \frac{\lambda}{1+\lambda^2}
 \\
\frac{\lambda}{1+\lambda^2} & \frac{1}{1+\lambda^2}
\end{pmatrix}.
\]
From this we easily find that \eqref{curves} admits a unique
solution given by
\begin{equation}\label{curves2}
z(\lambda) = \frac{1}{1+\lambda^2} \overline z -
\frac{\lambda}{1+\lambda^2} \overline z^\perp.
\end{equation}
When $\lambda> 0$ then equation \eqref{tpsi}  allows to find the
$t$-coordinate of the point $g(\lambda) = (z(\lambda),t(\lambda))\in
H_0$ having the property that
\begin{equation*}
d(g(\lambda),\overline g) = d(g(\lambda),\Pi_0).
\end{equation*}
Such $t$-coordinate is given by the equation 
\begin{equation*}
 t(\lambda) = \frac{|z(\lambda)|^2}{2}\ \Psi(\lambda)
= \frac{\Psi(\lambda)}{2(1+\lambda^2)} |\overline z|^2,
\end{equation*}
where $\Psi(\lambda)$ is as defined in \eqref{Psi}.

Summarizing, given $\overline g=(\overline z,0)\in \Pi_0$ and $\lambda>0$ the corresponding point $g(\lambda)$ that
admits $\overline g$ as the point that realizes its gauge distance to $\Pi_0$   is given by
\begin{equation}\label{gl}
g(\lambda) = (z(\lambda),t(\lambda)) =
\left(\frac{1}{1+\lambda^2}\ \overline z -
\frac{\lambda}{1+\lambda^2} \overline z^\perp,
\frac{\Psi(\lambda)}{2(1+\lambda^2)} |\overline z|^2\right).
\end{equation}

\begin{thrm}\label{T:inside-cha}
Let $\Om \subset \Hn$ be a (Euclidean) $C^{1,1}$ domain. Suppose that
the characteristic hyperplane $\Pi_0$ is tangent to $\p\Om$  at a non-characteristic point
$\overline g=(\overline z, 0)\in \partial\Om$ where $\overline z\not=0$ in such a way that $<\nu(\overline g), e_{2n+1}>=1$,
where $\nu(\overline g)$ denotes the unit inward normal to $\p\Om$ at $\overline g$ and $e_{2n+1}=(0,0,\dots,0,1)\in \mathbb{R}^{2n+1}$.
Then there exists a $\lambda_0>0$ depending on
$|\overline z|$ and  the $C^{1,1}$ character
of $\Om$, such that for  every $0<\lambda<\lambda_0$, one has \[
B(g(\lambda),R(\lambda)) \subset  \Om,\ \ \ \p \Om \cap
\overline B(g(\lambda),R(\lambda)) = \{\overline g\}, \] where
$\{g(\lambda)\}_{0<\lambda<\lambda_0}$ is given by \eqref{gl} and $R(\lambda)= d(g(\lambda), \overline g)$.
\end{thrm}

\begin{proof}
To prove this lemma we locally describe the boundary of $\Om$ as a
graph over the  hyperplane $H_0$. This means that we can find $r_0>0$
sufficiently small, and a $C^{1,1}$ function $$\phi :\left \{z\in
\R^{2n}\mid |z-\overline z|<r_0 \right \} \to \R,$$
such that $\phi(\overline
z) = 0$, $D\phi(\overline z) = 0$, and for which
$$\p \Om \cap \left\{(z,t)\in \Hn\mid |z-\overline z|<r_0, |t|<r_0^2\right \}$$
is given by
$$\left\{(z,t)\in \Hn \mid t= \phi(z),\ |z-\overline z|<r_0, |t|<r_0^2  \right\}, $$
whereas, by the assumption $<\nu(\overline g), e_{2n+1}>=1>0$, the set
$$\Om \cap \left \{(z,t)\in \Hn\mid |z-\overline z|<r_0, |t|<r_0^2 \right\}$$
is given by
$$\left \{(z,t)\in \Hn \mid t> \phi(z),\ |z-\overline z|<r_0, |t|<r_0^2 \right \}.$$
Notice from \eqref{vf2} that $X (\phi-t)(\overline z,0) = D\phi(\overline z)
+ \overline z^\perp/2 =  \overline z^\perp/2 \not= 0$, thanks to
the assumption $\overline z\not= 0$. Notice also that in view of the
assumption that $\phi\in C^{1,1}$ the graph of $\phi$ is contained
between two paraboloids, i.e., there exists a constant $A>0$ such that
\begin{equation}\label{parabola}
- A |z -\overline z|^2 \leq  \phi(z) \leq\ A |z -\overline z|^2,\ \ \ \text{for every}\  |z-\overline z| < r_0. \end{equation}

Next, consider the  ball $B(g(\lambda),R(\lambda))$ centered at
$g(\lambda)$ and with radius $R(\lambda) = d(g(\lambda),\overline
g)$. We note that
\begin{equation*}
R(\lambda) =
\frac{\psi(\lambda)}{\sqrt{1 + \lambda^2}} |\overline z| =
\frac{\lambda}{(1 + \lambda^2)^{1/4}} |\overline z|.
\end{equation*}

As it was proved before such ball is tangent to $H_0$ at $\overline g$. Its
boundary is described by the equation
\begin{equation}\label{ovaloid}
|z-z(\lambda)|^4 + 16 \left(t - t(\lambda) +
\frac{1}{2}<z(\lambda)^\perp,z>\right)^2 = R(\lambda)^4,
\end{equation}
or equivalently
\[
\left|t - t(\lambda) + \frac{1}{2}<z(\lambda)^\perp,z>\right| =
\frac{1}{4} \sqrt{R(\lambda)^4 - |z-z(\lambda)|^4 }.
\]
Since at $\overline g$ we have
\[
t - t(\lambda) + \frac{1}{2}<z(\lambda)^\perp,z> = - t(\lambda) +
\frac{1}{2}<z(\lambda)^\perp,\overline z> = -
\frac{\lambda^3}{4(1+\lambda^2)} |\overline z|^2 < 0,
\]
a local description of $\p B(g(\lambda),R(\lambda))$ near $\overline
g$ is given by
\[
t  =  \Phi_\lambda(z) = t(\lambda) -
\frac{1}{2}<z(\lambda)^\perp,z> - \frac{1}{4} \sqrt{R(\lambda)^4 -
|z-z(\lambda)|^4 }.
\]
Such representation is valid for all points $(z,t)$ which are below
the horizontal plane $H_{g(\lambda)}$ passing through $g(\lambda)$.
Since the equation of such plane is given by \[ t  = t(\lambda) -
\frac{1}{2}<z(\lambda)^\perp,z>,
\]
it is clear from \eqref{ovaloid} that the projection onto $H_0$ of
the intersection of $B(g(\lambda),R(\lambda))$ with $H_{g(\lambda)}$
is given by the $2n$-dimensional Euclidean ball $B_e(z(\lambda),R(\lambda)) =
\{z \in \R^{2n}\mid |z - z(\lambda)| < R(\lambda)\}$. In view of
\eqref{parabola} it will thus suffice to show that
\begin{equation}\label{ovabpar}
\Phi_\lambda(z)\ >\ A |z - \overline z|^2\ ,\ \ \  \text{for every}\
|z - z(\lambda)|<R(\lambda)\ .
\end{equation}

Now, for every $z$ such that $|z - z(\lambda)|<R(\lambda)$ there
exist $0\leq s \leq 1$ and $\omega$ such that $|\omega|=1$ for which
$z =  z(\lambda) + s R(\lambda) \omega$. We thus have
\[
|z - \overline z|^2 = |z(\lambda) - \overline z|^2 + s^2
R(\lambda)^2 + 2 s R(\lambda) <z(\lambda) - \overline z, \omega>.
\]
From \eqref{gl} we obtain
\begin{equation}\label{zlambdaz}
z(\lambda) - \overline z = - \frac{\lambda}{1+\lambda^2}
(\overline z^\perp + \lambda \overline z),
\end{equation}
and therefore
\begin{equation}\label{zlambdaz2}
|z(\lambda) - \overline z|^2\ =\ \frac{\lambda^2}{1+\lambda^2}
|\overline z|^2\ .
\end{equation}
Substituting \eqref{zlambdaz}, \eqref{zlambdaz2} in the above
equation we find
\begin{align}\label{zzbar}
|z - \overline z|^2 & = \frac{\lambda^2}{1+\lambda^2} |\overline
z|^2 + s^2 \frac{\lambda^2}{(1 + \lambda^2)^{1/2}} |\overline z|^2
\\
& - 2 s \frac{\lambda^2}{(1+\lambda^2)(1 + \lambda^2)^{1/4}}
|\overline z|^2 <\frac{\overline z^\perp}{|\overline z|} + \lambda
\frac{\overline z}{|\overline z|}, \omega>.\notag
\end{align}
Next, we recall that \[ t(\lambda) =
\frac{\Psi(\lambda)}{2(1+\lambda^2)} |\overline z|^2 =
\frac{\lambda (1 + \frac{\lambda^2}{2})}{2(1+\lambda^2)} |\overline
z|^2.
\]
Keeping in mind \eqref{gl} we find
\[
<z(\lambda)^\perp,z> = s R(\lambda) <z(\lambda)^\perp,\omega> =
\frac{s R(\lambda) |\overline z|}{1+\lambda^2} <\frac{\overline
z^\perp}{|\overline z|} + \lambda \frac{\overline z}{|\overline z|},
\omega>.
\]
From the latter two equations and from \eqref{zzbar} we conclude
that proving \eqref{ovabpar} is equivalent to proving that for every
$A>0$ there exists $\lambda(A)>0$ such that for every
$0<\lambda<\lambda(A)$, and for every $0\leq s\leq 1$ and $|\omega|
= 1$,
\begin{align}\label{ki}
&\frac{\lambda (1 + \frac{\lambda^2}{2})|\overline z|^2}{2(1+\lambda^2)}
 - \frac{s \lambda |\overline z|^2}{2(1+\lambda^2)(1
+ \lambda^2)^{1/4}} <\frac{\overline z^\perp}{|\overline z|} +
\lambda \frac{\overline z}{|\overline z|}, \omega> - \frac{\lambda^2
|\overline z|^2 (1 - s^4)^{1/2}}{4 (1 + \lambda^2)^{1/2}} 
\\
&  > A \left\{\frac{\lambda^2 |\overline z|^2}{1+\lambda^2}  +
\frac{s^2 \lambda^2 |\overline z|^2}{(1+\lambda^2)^{1/2}}   -
\frac{2 s \lambda^2 |\overline z|^2}{(1+\lambda^2)(1 +
\lambda^2)^{1/4}} <\frac{\overline z^\perp}{|\overline z|} + \lambda
\frac{\overline z}{|\overline z|}, \omega> \right\}. \notag
\end{align}

Establishing \eqref{ki} is in turn equivalent to proving

\begin{align}\label{ki2}
&2 + \lambda^2 - \frac{2s (1 - 4 A \lambda)}{(1 + \lambda^2)^{1/4}}
<\frac{\overline z^\perp}{|\overline z|} + \lambda \frac{\overline
z}{|\overline z|}, \omega> - \lambda (1 + \lambda^2)^{1/2} (1 -
s^4)^{1/2}\\
&  > 4 A \lambda\ \big[1 + s^2 (1+\lambda^2)^{1/2}\big].\notag
\end{align}

At this point we let $\lambda(A) = \frac{1}{4A}$, then it is clear that
$1-4A\lambda>0$ for every $0<\lambda<\lambda(A)$. Since
$\left|\frac{\overline z^\perp}{|\overline z|} + \lambda
\frac{\overline z}{|\overline z|}\right| = (1 + \lambda^2)^{1/2}$,
by Cauchy-Schwarz inequality it is clear that, provided that
$0<\lambda<\lambda(A)$,  for \eqref{ki2} to hold it suffices to
have
\begin{align*}
 2 + \lambda^2 - 2s (1 - 4 A \lambda)(1 + \lambda^2)^{1/4} -
\lambda (1 + \lambda^2)^{1/2} (1 - s^4)^{1/2}
 > 4 A \lambda\ \big[1 + s^2 (1+\lambda^2)^{1/2}\big].\notag
\end{align*}
or equivalently,
\begin{align}\label{ki3}
& 2 + \lambda^2 - 2s (1 + \lambda^2)^{1/4} - \lambda (1 +
\lambda^2)^{1/2} (1 - s^4)^{1/2} \\
& >\ 4 A \lambda\ \big [1 + s^2 (1+\lambda^2)^{1/2} - 2 s (1 +
\lambda^2)^{1/4}\big ].\notag
\end{align}
Since the quantity in the square brackets in right-hand side of
\eqref{ki3} is positive (it is a square), considering the fact that
$4A\lambda <1$, for \eqref{ki3} to hold it suffices that the
inequality
\[
2 + \lambda^2 - 2s (1 + \lambda^2)^{1/4} - \lambda (1 +
\lambda^2)^{1/2} (1 - s^4)^{1/2}\ \geq\ 1 + s^2 (1+\lambda^2)^{1/2}
- 2 s (1 + \lambda^2)^{1/4},
\] does hold for every $0<\lambda<\lambda(A)$ and every $0\leq s
\leq 1$. This inequality, however, is equivalent to the inequality
\[
\lambda (1 - s^4)^{1/2} + s^2  \leq (1 + \lambda^2)^{1/2}.
\]
The validity of this latter inequality now follows by applying
Cauchy-Schwarz inequality to the vectors $a = (\lambda,1),\  b = ((1 -
s^4)^{1/2},s^2)$, and noting that $|a| = (1 + \lambda^2)^{1/2},\  |b|= 1$.

\end{proof}

In the next lemma we  establish the quasi-segment property with respect to the gauge distance  along the path $g(\lambda)$.

\begin{lemma}\label{L:edfb}
Let $g_0 = (z_0,t_0)\in H_0$ with $z_0\not= 0$, and $\{g(\lambda)\}_{\lambda\geq 0}\ =\ \{(z(\lambda),t(\lambda))\}_{\lambda\geq 0}$
be as in \eqref{gl} with $\overline g=\overline g_0 = (z_0+ \lambda_0 z_0^{\perp},0)$ and $\lambda_0=G\left(\frac{2t_0}{|z_0|^2}\right)$.
Then there exists $\overline \lambda>0$ such that if
$\lambda_0<\lambda_1<\overline \lambda$ then
\[
d(g(\lambda_1),\overline g_0) - d(g(\lambda_1),g_0) \geq \frac{1}{2}
d(g_0,\overline g_0).
\]
\end{lemma}

\begin{proof}
From Lemma \ref{L:dcp} we see  that $\overline g$ is the point that  realizes the distance of $g_0
= (z_0,t_0)$ to $\Pi_0$. Moreover, by the hypothesis we have
\begin{equation}\label{connect}
\lambda_0 = G\left(\frac{2t_0}{|z_0|^2}\right),\ \ \text{or
equivalently}\ \  t_0 = \frac{|z_0|^2}{2}\ \Psi(\lambda_0).
\end{equation}

Replacing $\overline z$ with $ z_0\ +\ \lambda_0\ z^\perp_0\ $ in \eqref{curves2} we find for the
corresponding $g(\lambda) = (z(\lambda),t(\lambda))$
\begin{equation}\label{curves3}
\begin{cases}
z(\lambda) = \alpha(\lambda) z_0  + \beta(\lambda) z_0^\perp,
\\
t(\lambda) = \gamma(\lambda) |z_0|^2,
\end{cases}
\end{equation}
where
\begin{equation}\label{ab}
\begin{cases}
\alpha(\lambda)\ =\ \frac{1+ \lambda_0 \lambda}{1+\lambda^2}\ ,\ \ \
\beta(\lambda)\ =\ \frac{\lambda_0 - \lambda}{1 + \lambda^2}\ ,
\\
\gamma(\lambda)\ =\ \frac{1}{2}\ \frac{1+\lambda_0^2}{1+\lambda^2}\
\Psi(\lambda)\ .
\end{cases}
\end{equation}

Notice that
\[
\alpha(\lambda_0)\ =\ 1\ ,\ \ \beta(\lambda_0)\ =\ 0\ ,\ \
\gamma(\lambda_0)\ =\ \frac{1}{2}\ \Psi(\lambda_0)\ ,
\]
and so \eqref{curves3} and \eqref{connect} give
\begin{equation}\label{weareok}
g({\lambda_0})\ =\ (z({\lambda_0}),t({\lambda_0}))\ =\ (z_0,t_0)\ =\
g_0\ .
\end{equation}

Also notice that when $\lambda \to 0^+$, we have (recall that
$\Psi(\lambda)\to 0$ as $\lambda \to 0^+$)
\[
\alpha(\lambda)\ \to\ 1\ ,\ \ \beta(\lambda)\ \to\ \lambda_0\ ,\ \
\gamma(\lambda)\ \to\ 0\ ,\  \text{as}\ \lambda \to 0^+\ .
\]

We thus  have \begin{equation}\label{gzero} g(0)\ =\ \overline
g_0\ =\ (z_0 +\lambda_0 z_0^\perp,0)\ .
\end{equation}

 Furthermore,
\begin{equation*}
|z(\lambda)|\ =\ \sqrt{\alpha(\lambda)^2 + \beta(\lambda)^2}\ |z_0|\ ,
\end{equation*}
and that when $0< \lambda_0 <\lambda$, one has
\begin{equation*}
\alpha(\lambda)^2 + \beta(\lambda)^2\ =\ \frac{1 + \lambda_0^2
\lambda^2 + \lambda_0^2 + \lambda^2}{(1 + \lambda^2)^2}\ <\ 1\ .
\end{equation*}

We set henceforth
\[
\rho\ =\ d(g_0,\overline g_0)\ = \ d(g_0,\Pi_0) =\ d((z_0,t_0),\Pi_0)\ .
\]

 Notice that from \eqref{fd} we obtain
\begin{equation}\label{fd2}
\rho\  =\ \left(\lambda_0^4 |z_0|^4 +
16\left(t_0 - \frac{\lambda_0}{2}|z_0|^2\right)^2\right)^{1/4}\ .
\end{equation}

On the other hand, using \eqref{cc3}  we find \[ 16 \left(t_0 -
\frac{\lambda_0}{2} |z_0|^2\right)^2\ =\ |z_0|^4 \lambda_0^6\ ,
\]
and so \eqref{fd2} gives
\begin{equation}\label{fd3}
\rho\ =\ |z_0| \ \lambda_0 (1 + \lambda_0^2)^{1/4}\ .
\end{equation}

If we introduce the strictly increasing function
\[
\psi(s)\ =\ s (1 + s^2)^{1/4}\ ,\ \ \ s\geq 0
\]
as in \eqref{lp}, then it is clear from \eqref{fd3} that
\begin{equation}\label{fd4}
\lambda_0\ =\ \psi^{-1}\left(\frac{\rho}{|z_0|}\right)\ .
\end{equation}

At this point we fix a real number $\lambda_1 >\lambda_0$ and
call $g_1 = g({\lambda_1})$. Clearly, by the way we have constructed the path $\lambda \to
g(\lambda)$, the point $\overline g_1 = \overline g(\lambda_1)$
which realizes the distance of $g_1$ to $\Pi_0$ coincides with
$\overline g_0$. We have in fact
\[
\begin{cases}
\alpha(\lambda) - \lambda \beta(\lambda)\ =\ 1
\\
\lambda \alpha(\lambda) + \beta(\lambda)\ =\ \lambda_0\ ,
\end{cases}
\]
and this gives
\begin{equation*} 
z(\lambda) + \lambda
z(\lambda)^\perp\ =\ z_0 + \lambda_0 z_0^\perp\ ,\ \ \ \text{for
every}\ \lambda \geq 0\ .
\end{equation*}

We  now set
\[
\tilde g(\lambda)\ =\ g({\lambda_1 - \lambda})\ ,\ \ \ \ 0\leq
\lambda \leq \lambda_1\ ,
\]
and define a function $\phi_{\lambda_1}:[0,\lambda_1] \to [0,\infty)$ by letting
\begin{equation*}
\phi_{\lambda_1}(\lambda)\ =\ d(\tilde g(\lambda),g_1)\ =\
N(g_1^{-1} \tilde g(\lambda))\ .
\end{equation*}

Notice that if $\tilde \lambda_0 = \lambda_1 - \lambda_0$, then we
have
\[
\phi_{\lambda_1}(\tilde \lambda_0)\ =\ d(g({\lambda_0}),g_1)\ =\
d(g_0,g_1)\ ,
\]
where in the last equality we have used \eqref{weareok}. On the
other hand, from \eqref{gzero} we have
\[
\phi_{\lambda_1}(\lambda_1)\ =\ d(\tilde g({\lambda_1}),g_1)\ =\
d(\overline g_0,g_1)\ .
\]

From the mean value theorem we thus obtain for some $\tilde
\lambda_0 < \lambda^* < \lambda_1$
\begin{eqnarray*}
d(g_1,\overline g_0) - d(g_1,g_0)\ &=& \ \phi_{\lambda_1}(\lambda_1) -
\phi_{\lambda_1}(\tilde \lambda_0)\ =\ (\lambda_1 - \tilde
\lambda_0) \phi'_{\lambda_1}(\lambda^*) \\
\ &=& \ \lambda_0
\phi'_{\lambda_1}(\lambda^*)\ =\
\psi^{-1}\left(\frac{\rho}{|z_0|}\right)
\phi'_{\lambda_1}(\lambda^*)\ ,
\end{eqnarray*}
where in the last equality we have used \eqref{fd4}. Our goal is to
show that there exists $\overline \lambda>0$ sufficiently small, and
$C>0$, such that for all $0<\lambda_0<\lambda_1<\overline \lambda$ we have
\begin{equation}\label{ebb}
\psi^{-1}\left(\frac{\rho}{|z_0|}\right)  \phi'(\lambda^*)\ \geq\
\frac{1}{2}\ \rho\ .
\end{equation}

Since from \eqref{psinv2} we have
\[
\psi^{-1}\left(\frac{\rho}{|z_0|}\right)\ =\  \frac{\rho}{|z_0|} (1
+ o(1))\ ,\ \ \ \text{as}\ \frac{\rho}{|z_0|} \to 0^+\ ,
\]
we see from \eqref{fd3} that there exists $\overline r>0$ such that
\begin{equation}\label{psi}
\psi^{-1}\left(\frac{\rho}{|z_0|}\right)\ \geq\ \frac{1}{2}\
\frac{\rho}{|z_0|}\ ,\ \ \ \text{provided that}\ \lambda_0 < \overline r\ . \end{equation}

To establish \eqref{ebb} it will thus be enough to show that there
exists $\overline \lambda>0$ such that
\begin{equation}\label{phi'}
\phi'_{\lambda_1}(\lambda^*)\ \geq\ |z_0|\ >\ 0\ ,\ \ \text{for
every}\ \lambda_1 - \lambda_0 < \lambda^* < \lambda_1 \leq \overline
\lambda\ .
\end{equation}

Since from $\phi_{\lambda_1}(\lambda) =  N(g_1^{-1} g(\lambda_1 -
\lambda))$, and \[ \phi'_{\lambda_1}(\lambda^*)\ =\ -\
\frac{d}{d\lambda} N(g_1^{-1} g(\lambda)) \big|_{\lambda = \lambda_1
- \lambda^*}\ , \] it is clear that it will suffice to show that
there exist $\overline \lambda>0$ such that, if we set
$\Phi_{\lambda_1}(\lambda) = N(g_1^{-1} g(\lambda))$,
\[
\Phi'_{\lambda_1}(\lambda)\ \leq\ -\ |z_0|\ ,\ \ \text{for every} \
0\leq \lambda \leq \lambda_0 < \lambda_1 \leq \overline \lambda\ .
\]

Now we have from \eqref{curves3}
\begin{align}\label{comp}
g_1^{-1} g(\lambda) & = (- z_1, - t_1) (\alpha(\lambda) z_0 +
\beta(\lambda) z_0^\perp , \gamma(\lambda)|z_0|^2)
\\
& =\ \big((\alpha(\lambda) - \alpha(\lambda_1)) z_0 +
((\beta(\lambda) - \beta(\lambda_1)) z_0^\perp, (\gamma(\lambda) -
\gamma(\lambda_1))|z_0|^2  \notag\\
& +\ \frac{1}{2}<\alpha(\lambda) z_0 + \beta(\lambda) z_0^\perp,
z_1^\perp>\big) \notag
\\
& =\ \big((\alpha(\lambda) - \alpha(\lambda_1)) z_0 +
((\beta(\lambda) - \beta(\lambda_1)) z_0^\perp, \big(\gamma(\lambda)
-
\gamma(\lambda_1))  \notag\\
& +\ \frac{1}{2}\big(\beta(\lambda) \alpha(\lambda_1) -
\alpha(\lambda) \beta(\lambda_1)\big)\big) |z_0|^2\big) \notag
\end{align}

Using \eqref{comp} we find
\begin{align*}
\Phi_{\lambda_1}(\lambda)\ =\ \phi(\lambda_1 - \lambda) & =\
\bigg\{\bigg((\alpha(\lambda) - \alpha(\lambda_1))^2 +
(\beta(\lambda) - \beta(\lambda_1))^2
\bigg)^2  \\
& +\ 16\ \bigg( \gamma(\lambda) - \gamma(\lambda_1) +\
\frac{1}{2}\big(\beta(\lambda) \alpha(\lambda_1) - \alpha(\lambda)
\beta(\lambda_1)\big)\big)\bigg)^2 \bigg\}^{1/4}\ |z_0|\ , \notag
\end{align*}
and thus
\begin{align}\label{two}
\Phi'_{\lambda_1}(\lambda)\ & =\ |z_0|\
\bigg\{\bigg((\alpha(\lambda) - \alpha(\lambda_1))^2 +
(\beta(\lambda) - \beta(\lambda_1))^2
\bigg)^2  \\
& +\ 16\ \bigg( \gamma(\lambda) - \gamma(\lambda_1) +\
\frac{1}{2}\big(\beta(\lambda) \alpha(\lambda_1) - \alpha(\lambda)
\beta(\lambda_1)\big)\big)\bigg)^2 \bigg\}^{- 3/4}\notag\\
& \times 4\ \bigg\{\bigg((\alpha(\lambda) - \alpha(\lambda_1))^2 +
(\beta(\lambda) - \beta(\lambda_1))^2 \bigg)\notag \\
&\times \bigg((\alpha(\lambda)
- \alpha(\lambda_1)) \alpha'(\lambda) + (\beta(\lambda) -
\beta(\lambda_1)) \beta'(\lambda)\bigg) \notag\\
& +\ 8\ \bigg( \gamma(\lambda) - \gamma(\lambda_1) +\
\frac{1}{2}\big(\beta(\lambda) \alpha(\lambda_1) - \alpha(\lambda)
\beta(\lambda_1)\big)\big)\bigg)\notag \\
&\times \bigg(\gamma'(\lambda) +
\frac{1}{2} (\beta'(\lambda) \alpha(\lambda_1) - \alpha'(\lambda)
\beta(\lambda_1))\bigg)\bigg\} \ . \notag
\end{align}

Recalling that $\alpha(0) = 1$, $\beta(0) = \lambda_0$, $\gamma(0) =
0$, $\alpha'(0) = \lambda_0$, $\beta'(0) = - 1$, and $\gamma'(0) =
1/2$, we obtain from \eqref{two}, after some crucial cancellations
\begin{align}\label{three}
\Phi'_{\lambda_1}(0)\ & =\ |z_0|\ \bigg\{\bigg((1 -
\alpha(\lambda_1))^2 + (\lambda_0 + |\beta(\lambda_1)|)^2
\bigg)^2  \\
& +\ 16\ \bigg( - \gamma(\lambda_1) +\ \frac{1}{2}\big(\lambda_0
\alpha(\lambda_1) +
|\beta(\lambda_1)|\big)\big)\bigg)^2 \bigg\}^{- 3/4}\notag\\
& \times 4\ \bigg\{-\ \bigg((1 - \alpha(\lambda_1))^2 + (\lambda_0 +
|\beta(\lambda_1)|)^2 \bigg) \bigg(\lambda_0 \alpha(\lambda_1)) + |\beta(\lambda_1)|)\bigg) \notag\\
& +\ 4\ \bigg(- \gamma(\lambda_1) +\ \frac{1}{2}\big(\lambda_0
\alpha(\lambda_1) + |\beta(\lambda_1)|\big)\big)\bigg) \bigg(1 -
\alpha(\lambda_1) + \lambda_0 |\beta(\lambda_1)|\bigg)\bigg\} \notag
\end{align}

At this point we observe the following formulas which follow from
\eqref{ab} after some elementary computations
\[
(1 - \alpha(\lambda_1))^2 + (\lambda_0 + |\beta(\lambda_1)|)^2\ =\
\lambda_1^2\ \frac{1 + \lambda_0^2 \lambda_1^2 + \lambda_0^2 +
\lambda_1^2}{(1+\lambda_1^2)^2}\ .
\]
\[
\lambda_0\ \alpha(\lambda_1)\ +\ |\beta(\lambda_1)|\ =\ \lambda_1\
\frac{1+\lambda_0^2}{1+\lambda_1^2}\ ,
\]
\[
\gamma(\lambda_1)\ =\ \frac{\lambda_1}{2}\ \left(1 +
\frac{\lambda_1^2}{2}\right)\ \frac{1+\lambda_0^2}{1+\lambda_1^2}\ ,
\]
\[
-\ \gamma(\lambda_1)\ +\ \frac{1}{2}\ (\lambda_0\ \alpha(\lambda_1)\
+\ |\beta(\lambda_1)|)\ =\ -\ \frac{\lambda_1^3}{4}\
\frac{1+\lambda_0^2}{1+\lambda_1^2}\ ,
\]
\[
1 - \alpha(\lambda_1) + \lambda_0 |\beta(\lambda_1)|\ =\
\frac{\lambda_1^2 - \lambda_0^2}{1 + \lambda_1^2}\ .
\]

Substituting these equations in \eqref{three} we finally obtain
\begin{align*}
\Phi'_{\lambda_1}(0)\ & =\ - 4\ |z_0|\ (1 + \lambda_0^2)\ \frac{1 +
\lambda_0^2 \lambda_1^2 + \lambda_0^2 + \lambda_1^2 +
(1+\lambda_1^2)(\lambda_1^2-\lambda_0^2)}{[(1 + \lambda_0^2
\lambda_1^2 + \lambda_0^2 + \lambda_1^2)^2 + \lambda_1^2 (1 +
\lambda_0^2)^2 (1 +\lambda_1^2)^2]^{3/4}}
\\
& =\ - 4\ |z_0|\  \frac{(1 + \lambda_0^2) (1 + \lambda_1^2)^2}{[(1 +
\lambda_0^2 \lambda_1^2 + \lambda_0^2 + \lambda_1^2)^2 + \lambda_1^2
(1 + \lambda_0^2)^2 (1 +\lambda_1^2)^2]^{3/4}}\ . \notag
\end{align*}

Keeping in mind that $0 < \lambda_0 < \lambda_1$, we see that if
$\lambda_1 \to 0^+$, then
\[
\frac{(1 + \lambda_0^2) (1 + \lambda_1^2)^2}{[(1 + \lambda_0^2
\lambda_1^2 + \lambda_0^2 + \lambda_1^2)^2 + \lambda_1^2 (1 +
\lambda_0^2)^2 (1 +\lambda_1^2)^2]^{3/4}}\to\ 1\ .
\]

It is thus clear that there exists $\overline \lambda >0$ such that
\begin{equation}\label{lb}
\Phi'_{\lambda_1}(0)\ \leq\ -\ \frac{1}{2}\ |z_0|\ ,\ \ \
\text{provided that} \ 0 \leq \lambda_1< \overline \lambda\ .
\end{equation}

By continuity, \eqref{psi} and \eqref{lb}, if we choose $\overline \lambda>0$ sufficiently
small, we  achieve \eqref{phi'}, thus completing the proof.

\end{proof}

\section{Proof of Theorems \ref{T:eb} and \ref{T:cp}}\label{S:eb}

With the preliminary work developed in Sections \ref{S:qs1} and \ref{S:qs2}, we are finally able to establish Theorems \ref{T:eb} and \ref{T:cp}.

\begin{proof}[Proof of Theorem \ref{T:eb}]
We only discuss the case $p<Q$, leaving it to the reader to provide the appropriate modifications for the case $p=Q$. By the uniform Harnack inequality in \cite{CDG1} we know that $u>0$ in $\Om$. Let $g_0\in \p \Om$ which is not a characteristic point of $\p \Om$. For $0<r< \frac{d(g_0, \Sigma_{\Om})}{M} $, where $M$ is sufficiently large,  let $g\in \Om \cap B(g_0,r)$
and denote by $\overline g\in \p \Om$ a point such that $d(g,\p \Om)
= d(g,\overline g)$. We observe right away that
\begin{equation}\label{r}
d(g,\overline g)  < r. \end{equation}

One has in fact $d(g,\overline g) = d(g,\p \Om) \leq d(g,g_0) < r$. Moreover, $\overline g\in B(g_0, 2r)$ since
$d(\overline g, g_0)\leq d(\overline g, g) +d(g, g_0) < 2r$. If $M$ has been fixed large enough, this gives in particular that $\overline g\not\in \Sigma_\Om$.
We now let $\Pi(\overline g)$ be the hyperplane tangent to $\partial \Om$ at $\overline g$. There are two possibilities: either $\Pi(\overline g)$ has a characteristic point, or it does not.

\noindent \textbf{Case 1)} If
$\Pi(\overline g)$ has no characteristic point then it must have the form
$$\Pi(\overline g)=\{(x,y,t )\in \Hn \mid <a,x>+<b,y> + d=0 \},$$
where $a, b \in \Rn$ and $ d \in \R $. By a left-translation we may assume without loss of generality that $\Pi(\overline
g)$ takes the form
\[
\Pi_\omega=\{(x,y,t )\in \Hn \mid <a,x>+<b,y> =0 \},
\]
where $\omega = (a,b)\in \R^{2n}$. We emphasize that, thanks to \eqref{isometry}, left-translations do not change the gauge distance $d$, so that we can assume from the start that $\Pi(\overline g)$ is in the form $\Pi_\omega$.  Furthermore, without loss of generality we can assume that $|\omega|^2 = |a|^2 + |b|^2 = 1$.
And that the point $\overline g$ at which the domain $\Om$ and the hyperplane $\Pi_\omega$ touch is a non-characteristic point for $\p \Om$. 

\noindent \textbf{Case 2)} If instead $\Pi(\overline g)$ has a characteristic
point, then it is of the type
$$\Pi(\overline g)=\{(x,y,t )\in \Hn \mid <a,x>+<b,y> + c\, t + d=0 \},$$
where $c\not=0$. Notice that its characteristic point is 
\[
g_0 = \left(-\frac{2b}{c},\frac{2a}{c},-\frac{d}{c}\right)\in \Pi(\overline g).
\]
If in this case we left-translate  $\Pi(\overline g)$ by the point $g_0^{-1}$, it is easy to recognize that the hyperplane $\Pi(\overline g)$ is mapped into 
$\Pi_0=\{(x,y,t)\in \Hn\mid t>0\}$. 
Invoking Theorem \ref{L:vptb} and Lemma \ref{L:vp} in Case 1), and Theorem \ref{T:inside-cha} and Lemma \ref{L:edfb} in Case 2), we can construct
a quasi-segment $\{g(\lambda)\}_{0<\lambda<\overline \lambda}$ such that for every
$0<\lambda<\overline\lambda$,
\[
B(g(\lambda),R(\lambda))\ \subset \ \Om,\ \ \ \p \Om \cap
\overline B(g(\lambda),R(\lambda)) = \{\overline g\},
\]
where $R(\lambda)= d(g(\lambda), \overline g)$. 

\begin{figure}[h]
\begin{center}
\includegraphics[width=82mm]{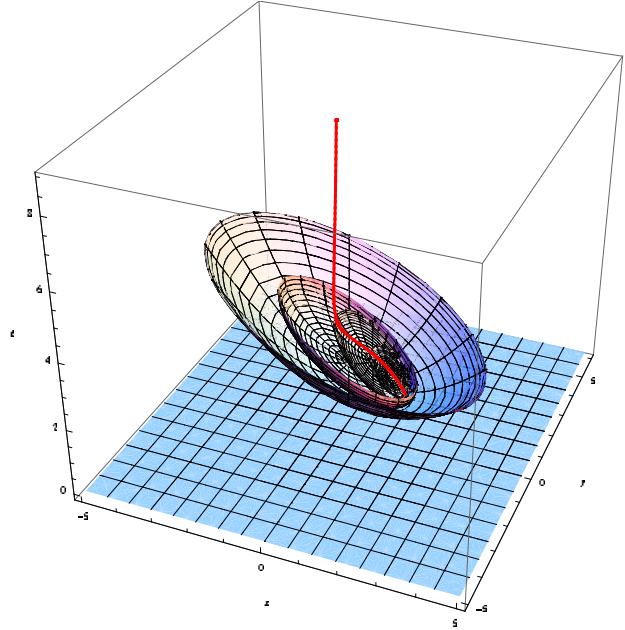}
\caption{Tangent balls and a quasi-segment at a non-characteristic point of the horizontal plane $t=0$.}
\end{center}
\end{figure}

Moreover, by \eqref{r} we may also assume that there exist $0<\lambda_0<\lambda_1<\overline\lambda$ such that
$g=g(\lambda_0)$ and  $R(\lambda_1)=2r$. Thus with  $g_1=g(\lambda_1)$ we have
$B(g_1,2r)\subset \Om$, $\overline g \in \p
B(g_1,2r)$ and
\begin{equation}\label{qsd}
d(g_1, \overline g)-d(g_1, g)\geq \frac{1}{2}d(g,\overline g),
\end{equation}
by the segment or quasi-segment property.

Since by Theorem \ref{T:MM}  any (Euclidean) $C^{1,1}$ domain in $\Hn$ is NTA, every boundary point of $\Om$ has attached to it a non-tangential corkscrew, see Definition \ref{D:NTA}. Now, it is clear that $g_1$ is a non-tangential corkscrew associated with
$\overline g$. Furthermore, for any $0<\epsilon<1$, $g'\in B(g_1,\epsilon \, r)$ 
and any $g''\in \p \Om$ one has
\[
2 r \leq d(g'',g_1) \leq d(g'',g') + d(g',g_1) \leq
d(g'',g') + \epsilon \, r,
\]
and also
\[ d(g', \p \Om) \leq d(g',\overline g) \leq d(g',g_1) +
d(g_1,\overline g)  =  d(g',g_1)+ 2r \leq (2 + \epsilon) r. \]

These inequalities imply that
\[
(2-\epsilon) r \leq d(B(g_1,\epsilon \, r),\p \Om)  \leq  (2 +
\epsilon) r.
\]

By the Harnack inequality in \cite{CDG1} we obtain for some
$C_\epsilon>0$
\begin{equation}\label{eb1}
\frac{u(g')}{u(g_1)} \geq C_\epsilon > 0,\ \ \text{for every}\
g'\in \overline B(g_1,\epsilon \, r).
\end{equation}

Since $g_1$ is a corkscrew for $\overline g$, and $A_r(g_0)$ is a
corkscrew for $g_0$, and $d(\overline g,g_0)<2r$, we see that there
exists a Harnack chain of $N$ balls (with $N$ independent of $r$)
connecting $g_1$ to $A_r(g_0)$. A repeated application of the
Harnack inequality thus gives \[ u(g_1)  \leq C^N u(A_r(g_0)).
\]

From this estimate and \eqref{eb1} we obtain
\begin{equation}\label{eb11}
\frac{u(g')}{u(A_r(g_0))} \geq C_\epsilon > 0,\ \ \text{for
every}\ g'\in \overline B(g_1,\epsilon \, r).
\end{equation}

We now consider the function
\[
f(g')\ =\ \frac{d(g',g_1)^{(p-Q)/(p-1)} - (2r)^{(p-Q)/(p-1)}}{(\epsilon \,
r)^{(p-Q)/(p-1)} - (2r)^{(p-Q)/(p-1)}}\ =\ \frac{
d(g',g_1)^{(p-Q)/(p-1)} - d(\overline
g,g_1)^{(p-Q)/(p-1)}}{(\epsilon \, r)^{(p-Q)/(p-1)} -
(2r)^{(p-Q)/(p-1)}}\]
on the  ring $B(g_1, 2r) \setminus \overline B(g_1,\epsilon \, r)$, and we
clearly have $f \equiv 1$ on $\p B(g_1,\epsilon \, r)$ and $f\equiv  0$
on $\p B(g_1, 2r)$. By \eqref{eb11} and the comparison principle
(Theorem \ref{T:ct}) we conclude that
\begin{equation}\label{eb2}
\frac{u(g')}{u(A_r(g_0))}\ \geq\ C\ f(g')\ ,\ \ \text{for every}\
g'\in \overline B(g_1, 2r) \setminus  B(g_1,\epsilon \, r)\ .
\end{equation}

Now, thanks to \eqref{qsd},  $g\in \overline B(g_1,2r) \setminus
B(g_1,\epsilon \, r)$ and hence we  obtain from \eqref{eb2}
\begin{equation*}
\frac{u(g)}{u(A_r(g_0))}\ \geq\ C\ f(g)\ .
\end{equation*}

It thus suffices to prove that
\[
f(g)\ \geq\ C\ \frac{d(g,\overline g)}{r} \ .
\]

Applying the mean value theorem to the function $h(s) =
s^{(p-Q)/(p-1)}$ with $\epsilon \, r\leq s \leq 2r$, and using the fact that $0<\ep<1$ we find
\[
f(g)\ \geq \ C(Q, p)\ \frac{d(g_1,\overline g) - d(g_1,g)}{r}\ .
\]

At this point we apply the quasi-segment property \eqref{qsd} to get the desired estimate.
This completes the proof of the theorem.

\end{proof}

\begin{proof}[Proof of Theorem \ref{T:cp}]
Combine Theorem \ref{T:eaADP} with Theorem \ref{T:eb}.

\end{proof}

\end{document}